\documentclass{amsart}

\usepackage{hyperref}
\usepackage{amssymb, amsmath, amsthm, amsfonts}
\usepackage{verbatim}
\usepackage{bbm}
\usepackage{enumerate}
\usepackage{dsfont}
\usepackage{upgreek}
\usepackage[mathscr]{eucal}
\usepackage{tikz}
\usepackage{subcaption}
\usepackage[normalem]{ulem}
\usepackage{graphicx}
\usepackage{enumitem}
\usepackage{float}

\usepackage[backgroundcolor=black!5,linecolor=black]{todonotes}
\usepackage{comment}

\usepackage{kotex}

\newtheorem{thm}{Theorem}[section]

\newtheorem{prop}[thm]{Proposition}

\newtheorem{lem}[thm]{Lemma}
\numberwithin{equation}{section}
\newtheorem{rem}{Remark}[section]

\newcommand{\Be}{\begin{equation}}
\newcommand{\Ee}{\end{equation}}
\newcommand{\dist}{\operatorname{dist}}
\newcommand\supp{\operatorname{supp}}
\newcommand\diam{\operatorname{diam}}

 \newcommand{\rd}{\mathrm d}

\newcommand{\e}{\mathrm{e}}

\date{\today}

\begin{document}

\author[J.~B. Lee]{Jin Bong Lee} \address[J.~B. Lee]{Research Institute of Mathematics, Seoul National University,
Seoul 08826, Republic of Korea} \email{jinblee@snu.ac.kr}

\author[S. Lee]{Sanghyuk Lee} \address[S. Lee]{Department of Mathematical Sciences and RIM, Seoul National University, Seoul 08826, Republic of  Korea} \email{shklee@snu.ac.kr}

\author[Roncal]{Luz Roncal} \address[Roncal]{BCAM -- Basque Center for Applied Mathematics, 48009 Bilbao, Spain, and 
Universidad del Pa{\'i}s Vasco / Euskal Herriko Unibertsitatea, 48080 Leioa, Spain, and Ikerbasque, Basque Foundation for Science, 48011 Bilbao, Spain} \email{lroncal@bcamath.org}


\title[Estimates over fractal time]{Strichartz and  local smoothing estimates for \\ the fractional Schr\"odinger equations \\  over fractal time}
 
\begin{abstract}  We obtain Strichartz-type estimates for the fractional Schrödinger operator $f \mapsto e^{it(-\Delta)^{\gamma/2}} f$ over a time set $E$ of fractal dimension. To obtain those  estimates capturing  fractal nature of $E$, we employ the notions  in the spirit of the Assouad dimension, such as, bounded Assouad characteristic and Assouad spectrum.  We also prove the estimate
$$
\| e^{it(-\Delta)^{\gamma/2}} f \|_{L_t^q(\mathrm{d}\mu; L_x^r(\mathbb{R}^d))} \le C \|f\|_{H^s},
$$
where $\mu$ is a measure satisfying an $\alpha$-dimensional growth condition.  In addition, we establish related inhomogeneous estimates and $L^2$ local smoothing estimates. A surprising feature of our work is that, despite dealing with rough fractal sets, we extend the known estimates for the fractional Schrödinger operators in a natural way, precisely consistent with the associated fractal dimensions.  
\end{abstract}

\maketitle

\section{introduction}
For a positive  number $\gamma$, we consider the operator 
\[  U^\gamma_t f(x)=   \e^{it(-\Delta)^{\gamma/ 2}} f (x):=(2\pi)^{-d} \int_{\mathbb R^d} \e^{i(\langle x, \xi\rangle+ t|\xi|^\gamma)} \widehat f(\xi) \,\mathrm{d}\xi, \] 
which is a solution to the initial value problem 
\[
i\partial_t u +  (-\Delta)^{\gamma/2} u=0, \quad u(0) = u_0.
\]
The case \(\gamma = 2\)  corresponds to the classical Schrödinger equation, which arises from the Brownian motion. For $\gamma \in (0, 2) \setminus \{1\}$,   the fractional Schrödinger equation is derived by applying Feynman's path integral framework to Lévy motion;  see  Laskin \cite{La1, La2}.

Space-time estimates, known as Strichartz estimates, play an important role in the study of  both linear and nonlinear equations by providing quantitative control of the solutions.   These estimates are formulated  in mixed norms in $x$ and $t$, whose exponents are required to satisfy  an \textit{admissible} condition.
 We say that a pair $(q,r)\in \mathbb [2,\infty]\times [2,\infty)$ is $\sigma$-admissible if 
\[  \frac1q\le \sigma \Big(\frac12-\frac1r\Big). \]
It is well known that 
\[
\| \e^{it(-\Delta)^{\gamma/2}} f \|_{L_t^q(\mathbb R; L^r_x (\mathbb R^d))} \le C \| f \|_{\dot{H}^s}, \quad s = \frac{d}{2} - \frac{d}{r} - \frac{\gamma}{q}
\]
holds  for sharp \(d/2\)-admissible \((q, r)\), provided that   \(\gamma \in \mathbb{R}_{>0} \setminus \{1\}\) (see, for example, \cite{COX}).
Here,  $\dot{H}^s$ denotes the homogeneous Sobolev space with norm given by \eqref{hom_sobolev}.  
For the wave operator $\e^{it(-\Delta)^{1/2}}$, the estimate holds for $\frac{d-1}{2}$-admissible $(q,r)$.
We recommend \cite{KPV, COX, ChoLee, HongSire} and references therein for the previous studies on the Strichartz estimate for the fractional Schr\"odinger equation.

Let $E \subset \mathbb{R}$ be a bounded set. 
In this note,  we are concerned with the Strichartz type estimates  over $\mathbb R^d\times E$, where  the time variable is restricted (from 
$\mathbb R$) to a subset  $E$ of fractal dimension, and the Lebesgue measure $\mathrm dt$  is replaced by a general measure $\mu$ supported on 
$E$ satisfying certain conditions. This approach enables us to investigate the temporal behavior of the solution  on the  lower-dimensional set $E$.  Most typical examples of such sets are  the  Cantor sets. 

Those kinds of estimates have recently appeared in the  works \cite{AHRS, BRRS, RS, LRZZ} (see \eqref{discrete} below), where spherical maximal functions defined via the supremum over a fractal set of dilation parameters were studied. More precisely, Strichartz-type estimates for the wave operator were obtained, relating the size (or dimension) of $E$ 
  to the smoothness of the input function.

Let $N(E, \delta)$ denote the minimal number of intervals of length $\delta$ needed to  cover $E$.
The (upper) Minkowski dimension $\dim_M E $ of $E$ is given by the infimum of $\beta>0$ such that
\[ 
	N(E, \delta) \leq C(\varepsilon) \delta^{-\beta-\varepsilon} 
\]
for every  $\epsilon>0$.   
This notion is particularly useful for quantifying the size of a lower-dimensional set $E$  via coverings. 

However,  as shown  in the previous works 
\cite{AHRS, BRRS, RS},   the Minkowski dimension alone  is not enough to capture the dispersive nature of the wave operator. 
This is due to the multi-scale phenomena  associated to frequency localization, i.e., Littlewood--Paley projection.
As to be seen later, by using  the Minkowski dimension alone, it is difficult to obtain a precise measurement of quantities  in relative scale  with respect to the size of the frequency of a given function.

In this paper, we mainly make use of the Assouad dimension, which  provides a  more effective framework for obtaining quantitative estimates over fractal sets.  
The Assouad dimension $\dim_A E$  is defined by
\[
\dim_A E = \inf \{a > 0 : \exists c > 0 \text{ s.t. } \forall I, \delta \in (0,|I|), N(E \cap I, \delta) \leq c\delta^{-a}|I|^a\},
\]
where $I$ runs over subintervals of $[0, 1]$.  It is clear that $ \dim_M E \leq \dim_A E $. Although the definitions of the Assouad and Minkowski  dimensions are similar,  the inequality  can be strict. For example,  consider a set $ E_a = \{n^{-a} : n=1,2,3,\dots \}$ for $0<a$. Then we have
$	 \dim_M(E_a)= 1/(1+a)< 1=\dim_A(E_a). $ 
We note that $\dim_M(E) = \dim_A(E)$ when $E$ is Ahlfors--David regular (see, for example,  \cite[Theorem 6.4.1]{Fra}).

\subsection{Homogeneous estimates}  

In order to get a precise quantitative estimate for the operator $U^\gamma_t$ over the fractal set $E$, the 
Assouad dimension alone is insufficient due to the nature of its  definition.   

Unlike the typical cases, we do not assume  that $E$ is contained in a fixed compact set. 
For $\alpha\in (0,1]$ and a  set $E \subset \mathbb R$, we define 
\begin{equation} 
\label{assouad}
 [E]_\alpha= 
 \sup_{0 < \delta } \sup_{\delta \leq |I| } \Big(\frac{\delta}{|I|} \Big)^{\alpha} N(E \cap I, \delta),
\end{equation}
where $I$ runs over any interval in $\mathbb R$.   
This may be regarded as a global version of  the notion of bounded $\alpha$-Assouad characteristic. We say that  $E$ has bounded $\alpha$-Assouad characteristic if $[E]_\alpha<\infty$.

Our first result concerns the estimates with frequency localization.  Let $\psi \in C_c^\infty(1/4,2)$.  For $j>0$, let $P_j$ denote the Littlewood--Paley projection operator defined by 
\Be
\label{l-p}
	 P_jf =(\psi(2^{-j} |\cdot|)  \widehat f \,)^\vee,\quad 
\Ee
and $P_{\le 0}={\rm Id}-\sum_{j=1}^\infty P_j$. Also, let $H^s$ denote the inhomogeneous Sobolev spaces of order $s$ with the norm 
\[  \|f\|_{H^s}^2=\int (1+|\xi|^2)^s |\widehat f(\xi)|^2\, \mathrm d\xi. \]  
 For simplicity, we write $L_x^r := L^r_x(\mathbb{R}^d)$ for the remainder of the paper.
For $\delta\in \mathbb R_{>0}$, we denote  $E(\delta)=\{ x: \dist(x, E)<\delta\}$.

 \begin{thm}  
  \label{thm_main}  
  Let  $\alpha\in (0,1]$ and $\gamma\in \mathbb R_{>0} \setminus \{1\}$, and  $2\le q,r \leq \infty$ with $(d/\alpha, q,r)\neq (2 ,2,\infty)$.  Suppose that $E$ is a bounded set with bounded $\alpha$-Assouad characteristic. 
 Then,  the estimate 
\Be \label{f-stri}    \sup_{j\ge 1}   \big\| U^\gamma_t P_j f\big\|_{L_t^q( E(2^{-\gamma j}); L_x^r)} + \big\| U^\gamma_t P_{\le 0} f\big\|_{L_t^q( E(1); L_x^r)}   \le C [E]_\alpha^\frac1q \|f\|_{H^s} \Ee
 holds   provided that $(q,r)$ satisfies 
 \Be 
 \label{con-ad} 
\frac d2 \Big(\frac 12 - \frac1r\Big) \ge  \frac \alpha q
 \Ee 
 and 
  \Be 
 \label{con-reg} 
 s\ge    s_\gamma (q,r):= \frac d2-\frac d r  - \frac {\gamma} q . 
 \Ee 
 \end{thm}

 If $\dim_A E =\alpha$,   then by a straightforward application of the definition of the Assouad dimension, one obtaians  the estimate \eqref{f-stri} provided that \eqref{con-ad} is replaced by a strict inequality.   
Using some specific sets such as Cantor sets,  one can also show that the result is basically sharp. (See  Section~\ref{subsec_nec}.) 
 
 When  $E = \mathbb R$, so $\alpha = 1$, the result recovers the classical Strichartz estimate (\cite{KeelTao}) via the Littlewood--Paley inequality and scaling. Indeed, 
 once we have  the local-in-time  estimate, the global-in-time estimate can also be deduced by a scaling argument using homogeneity. As $\alpha$ decreases, the admissible range \eqref{con-ad} of $(q, r)$ becomes larger. This is natural, as a smaller $\alpha$ indicates that control is required only on a smaller time set, thus allowing a broader range of admissible pairs.  At the extreme case  $\alpha = 0$, for example, when  $E$ is a singleton, i.e., $E = \{t_0\}$ for some $t_0\in \mathbb R$, the estimate \eqref{f-stri} holds for all $r, q \ge 2$, provided that \eqref{con-reg} is satisfied. This is, in fact, a straightforward consequence of the Sobolev embedding. 
 
 The underlying time set $E(2^{-\gamma j})$ is also natural in view of the uncertainty principle, since the temporal Fourier transform of $U^\gamma_t P_j f$ is supported in an interval of length comparable to $2^{\gamma j}$.   The condition \eqref{con-reg} is necessary for \eqref{f-stri} to hold, and so is  \eqref{con-ad}  when $s= d/2-d/r  - \gamma/q$. This will be shown near the end of Section~\ref{sec_hom}. 
 
  Let  
\[\mathcal E_j= \{ \tau \} \]  be  a maximally $2^{-\gamma j }$-separated subset of   $E$. That is to say, $|\tau - \tau'|\geq2^{-\gamma j }$ if $\tau, \tau'$ are distinct elements of $\mathcal E_{j}$,  and  for $\tau'\in  E\setminus  \mathcal E_j$  there is an element  $\tau\in \mathcal E_j$ such that $|\tau-\tau'|<2^{-\gamma j }$.  
The estimate \eqref{f-stri} is  equivalent to its discrete counterpart 
\Be  \label{discrete0}
\| U_{\tau}^\gamma P_j f\|_{\ell_{\tau}^q( \mathcal E_j  ;L^r_x)} \leq C 2^{(s+ \frac\gamma q)j} [E]_\alpha^\frac1q \|f\|_{L^2_x},\Ee
where $ \| a_\tau \|_{\ell_{\tau}^q}= (\sum_{\tau\in \mathcal E_j}    |a_\tau|^q)^\frac1q$. This can be justified easily using a locally constant property (Lemma~\ref{lem_250901}) and  will be exploited later to prove Theorem \ref{thm_main}.

When $\gamma = 1$ (i.e., $U_t^\gamma$ is the wave operator) and $E$ is  a bounded set, the estimates of a discrete form related to \eqref{discrete0} have already appeared in earlier works. 
Those estimates were  introduced to prove $L^p$–$L^q$ boundedness of the spherical maximal function for $d \ge 3$  \cite{AHRS}, and of the circular maximal function  \cite{RS} over fractal dilation sets, except for certain endpoint cases.  Recently,  variants of such estimates were studied by Wheeler \cite{Wh}, Beltran--Roos--Ruter--Seeger \cite{BRRS},  and Lee–Roncal–Zhang–Zhao \cite{LRZZ}, in connection with $L^p$ and $L^p$–$L^q$ local smoothing estimates for the wave operator over fractal time sets.
Following the argument in this paper, one can also obtain the corresponding result  for the wave operator; see Remark~\ref{wave-homo} below.
  
By $\dot H^s$, we denote the homogeneous Sobolev spaces of order $s$ with the norm 
\begin{align}\label{hom_sobolev}  \|f\|_{\dot H^s}^2=\int  |\xi|^{2s} |\widehat f(\xi)|^2 \,\mathrm d\xi. \end{align}
We obtain the following theorem, which gives an essentially scaling invariant estimate.  It should be noted that  $E$ is not assumed to be bounded. 

 \begin{thm}\label{homo-homo}   Let  $\alpha\in (0,1]$, $\gamma\in \mathbb R_{>0} \setminus \{1\}$, and  $2\le q,r \leq \infty$ with $(d/\alpha, q,r)\neq (2 ,2,\infty)$ and $\frac d r  + \frac {\gamma} q < d$.  Suppose that $E$ has bounded $\alpha$-Assouad characteristic.  Then,  the estimate 
\Be 
\label{f-stri-hs}   \sup_{-\infty< j <\infty}  \big\| U^\gamma_t P_j f\big\|_{L_t^q( E(2^{-\gamma j}); L_x^r)} \le C [E]_\alpha^\frac1q \|f\|_{\dot H^{s}}
\Ee
 holds   if $(q,r)$ satisfies \eqref{con-ad}  and $s=s_\gamma(q,r)$. 
 \end{thm}

Here, we exclude the case $\frac{d}{r} + \frac{\gamma}{q} \ge d$, since 
the space $\dot{H}^s$ does not admit the Schwartz class as a natural dense
subspace when $s \le -\frac{d}{2}$.

Although the estimates \eqref{f-stri} and  \eqref{f-stri-hs} give control over any  set $E$ in terms of its boundedness of $\alpha$-Assouad characteristic, it is less desirable in that the estimate is formulated via frequency localization.  
To avoid reliance on frequency localization,  we  assume that  $E$ is a set  on which a positive measure $\mu$ is supported. Instead of the bounded 
$\alpha$-Assouad characteristic condition, we  consider 
  \Be\label{alpha-dim}  
  \mu( (t-\rho, t+ \rho)) \le C\rho^\alpha
  \Ee
  for all $(t,\rho)\in \mathbb R\times (0,\infty)$ with $\alpha \in (0, 1]$.  By  $\langle \mu\rangle_\alpha$ we denote the best possible $C$ such that 
 \eqref{alpha-dim}  holds.  
  By  Frostman's  lemma (see, e.g., \cite[Chapter 8]{Mat})  \eqref{alpha-dim} implies 
 the Hausdorff dimension of $E$, $\dim_H E \ge \alpha$.

 With those assumptions, we obtain the following. 
   
 \begin{thm}\label{cor_main} 
 Let  $\alpha\in (0,1]$, $\gamma\in \mathbb R_{>0} \setminus \{1\}$, and  $2\le q,r \leq \infty$ with $(d/\alpha, q,r)\neq (2 ,2,\infty)$.
  Suppose that   $\mu$ has  bounded support and satisfies  \eqref{alpha-dim}. Then, we have 
 \Be \label{f-stri10}   \| U^\gamma_t f\|_{L_t^q( \mathrm{d}\mu ; L_x^r)} \le C\langle \mu\rangle_\alpha^\frac1q\|f\|_{H^s}\Ee
  provided that \eqref{con-ad} holds and 
  \begin{align}\label{con-reg2}
  	s\ge s_{\gamma\alpha}(q,r).
\end{align}
 \end{thm}

 It is unsurprising that Theorems \ref{thm_main} and \ref{cor_main} are formulated under the same hypothesis \eqref{con-ad}. Indeed, if $\mu$ possesses stronger regularity property, for example, if $\mu$ is an Ahlfors–David $\alpha$-regular measure, then in particular \eqref{alpha-dim} holds and this already implies that $\supp\mu$ has bounded $\alpha$-Assouad characteristic. This can be shown by an elementary argument. Since the measures $\mu$ in Theorem \ref{cor_main} may be of this type,  we obtain \eqref{f-stri} for $E=\supp\mu$ under the same condition \eqref{con-ad}.

The following theorem is a homogeneous version of Theorem \ref{cor_main}, where  $\supp \mu$ is not assumed to be bounded. 

\begin{thm}\label{cor_m} 
 Let  $\alpha\in (0,1]$, $\gamma\in \mathbb R_{>0} \setminus \{1\}$, and  $2\le q,r \leq \infty$ with $(d/\alpha, q,r)\neq (2 ,2,\infty)$ and $\frac d r  + \frac {\gamma} q < d$. 
  Suppose that   $\mu$  satisfies  \eqref{alpha-dim}. Then, we have 
 \Be \label{f-stri1}   \| U^\gamma_t f\|_{L_t^q( \mathrm{d}\mu ; L_x^r)} \le C\langle \mu\rangle_\alpha^\frac1q\|f\|_{\dot H^s}\Ee
  provided that \eqref{con-ad} holds and $s=s_{\gamma\alpha}(q,r)$. 
 \end{thm}

\subsection{Inhomogeneous estimates}  The adjoint operator of   $U^\gamma_t: L^2_x\to L_t^q( \mathrm{d}\mu ; L_x^r)$ is given by 
\[  F\to \int e^{-is(-\Delta)^{\gamma/2}} F(s)\,\mathrm d\mu(s), \]
where $F(s)=F(\cdot, s)$. 
In view of $TT^*$ argument 
(see Section~\ref{sec_hom}),  the boundedness of  $U^\gamma_t$ from   $L^2_x$ to $L_t^q( \mathrm{d}\mu ; L_x^r)$ is equivalent to the boundedness from $L_t^{q'}( \mathrm{d}\mu ; L_x^{r'})$ to $L_t^q( \mathrm{d}\mu ; L_x^r)$  of   the operator
\[ \int \mathrm{e}^{i(t - s)(-\Delta)^{\gamma/2}} F(s)\,\mathrm{d}\mu(s), \]
which is closely related to  the  solution  of the  inhomogeneous equation 
$
i\partial_t u +  (-\Delta)^{\gamma/2} u=F. 
$

 We now consider the estimate of the  form
\begin{align}\label{f-inhom}
	&\Big\| \int \mathrm{e}^{i(t - s)(-\Delta)^{\gamma/2}} F(s)\,\mathrm{d}\mu(s) \Big\|_{L_t^q(\mathrm{d}\mu; L_x^{r})} \le C \| F\|_{L_s^{\tilde{q}'}(\mathrm{d}\mu ; L_x^{\tilde{r}'})}.
\end{align}
This of course includes the usual inhomogeneous Strichartz estimate when $\mu$ is the Lebesgue measure.  
In view of the condition \eqref{alpha-dim} and a standard scaling heuristic, it is natural to impose a scaling condition 
\Be 
\label{scaling} 
\sigma_{\alpha}^\gamma(\tilde{r}, r, \tilde{q}, q) =0, 
\Ee
where 
\begin{align}
\label{con-inh-s}
	\sigma_{\alpha}^\gamma(\tilde{r}, r, \tilde{q}, q) &=d \Big(1 - \frac1{\tilde{r}} - \frac1r\Big) - \gamma  \Big(\frac{\alpha}{\tilde{q}} + \frac{\alpha}{q}\Big). 
\end{align}

The constraint \eqref{scaling}  can be justified by an exact scaling argument when $\mu$ is the Lebesgue measure (see, e.g., \cite{Vil}).   For a general measure, although a change of variables is not available as in the smooth setting, the assumption \eqref{alpha-dim} provides  a  scaling inequality as if the inequality  
\begin{align}\label{psd-scaling}
  \int h(t)\,\mathrm d\mu(t) \lesssim \lambda^{\alpha}    \int h(\lambda t) \,\mathrm d\mu(t), \quad(\lambda>0) 
\end{align}
held true for positive $\mu$-measurable functions.  (See Section~\ref{sec:2-1}.)  For certain nice measures, this can be upgraded to two-sided comparability, so that $\lesssim$ can be replaced by $\sim$. In such cases, via the typical scaling argument one can obtain \eqref{scaling}. 

\begin{figure}[t]
\centering
\begin{minipage}[t]{0.45\textwidth}
\centering
\begin{tikzpicture}[scale=4]
  % Axes
  \draw[->] (0,0) -- (1.05,0) node[right] {\( \frac{1}{\widetilde r} \)};
  \draw[->] (0,0) -- (0,1.05) node[above] {\( \frac{1}{r} \)};
  
  % Unit square boundary
  \draw[thick] (0,0) rectangle (1,1);

  % Dashed reference lines at 1/2
  \draw[densely dotted] (0.5,0) -- (0.5,1);
  \draw[densely dotted] (0,0.5) -- (1,0.5);
  \node at (0.5,-0.1) {\( \frac{1}{2} \)};
  \node at (-0.1,0.5) {\( \frac{1}{2} \)};

  % Coordinates
  \coordinate (A) at (0.5, 0.5);      % (1/2, 1/2)
  \coordinate (B) at (0.5, 0.3);    % (1/2, 3/8)
  \coordinate (C) at (0, 0);          % (0, 0)
  \coordinate (D) at (0.3, 0.5);   % (5/16, 1/2)
  \coordinate (P1) at (0.3, 0.3);       % Point on line
  \coordinate (P2) at (9/40, 3/8);      % upper intersection
  \coordinate (P3) at (3/8, 9/40);      % lower intersection

  % Fill and draw the quadrilateral
  \fill[blue!20] (A) -- (B) -- (P3) -- (P2) -- (D) -- cycle;
  \draw[blue, thick] (A) -- (B) -- (P3) -- (P2) -- (D) -- cycle;
  \draw[densely dotted] (P2) -- (C) ;
   \draw[densely dotted] (P3) -- (C) ;

  % Draw dots
  \filldraw (A) circle (0.01);
  \filldraw (B) circle (0.01);
  \draw (C) circle (0.01);
  \filldraw (D) circle (0.01);
  \draw (P2) circle (0.01);  
  \draw (P3) circle (0.01);
  
    % Draw black diagonal line segment (limited inside polygon)
  \draw[blue, thick] (P3) -- (P1) -- (P2);

  % Label the points with coordinates
  \node[right] at (B) {$A$};
  \node[below left] at (C) {$O$};
  \node[above] at (D) {$A'$};
   \node[below right] at (P3) {$C$};
   \node[above left] at (P2) {$C'$};

   \draw[black, densely dotted] (C) -- +(0.5,0.5);
   \draw[black, densely dotted] (P1) -- +(0.2, 0);
    \draw[black, densely dotted] (P1) -- +(0,0.2);
\end{tikzpicture}
\caption{$\frac{d}{2} >\alpha$}
\label{fig_lambda}
\end{minipage}
\hfill
\begin{minipage}[t]{0.45\textwidth}
\centering
\begin{tikzpicture}[scale=4]
  % Axes
  \draw[->] (0,0) -- (1.05,0) node[right] {\( \frac{1}{\widetilde r} \)};
  \draw[->] (0,0) -- (0,1.05) node[above] {\( \frac{1}{r} \)};
  
  % Unit square boundary
  \draw[thick] (0,0) rectangle (1,1);

  % Dashed reference lines at 1/2
  \draw[densely dotted] (0.5,0) -- (0.5,1);
  \draw[densely dotted] (0,0.5) -- (1,0.5);
  \node at (0.5,-0.1) {\( \frac{1}{2} \)};
  \node at (-0.1,0.5) {\( \frac{1}{2} \)};

  % Coordinates
  \coordinate (A) at (0.5, 0.5);      % (1/2, 1/2)
  \coordinate (B) at (0.5, 0.0);    % (1/2, 0)
  \coordinate (C) at (0, 0);          % (0, 0)
  \coordinate (D) at (0.0, 0.5);   % (0, 1/2)

  % Fill and draw the quadrilateral
  \fill[blue!20] (A) -- (B) -- (C) -- (D) -- cycle;
  \draw[blue, thick] (D) -- (A) -- (B);
  \draw[blue, densely dotted] (D) -- (C) -- (B);

  % Draw dots
  \filldraw (A) circle (0.01);
  \draw (B) circle (0.01);
  \draw (D) circle (0.01);

  % Label the points with coordinates
  \node[below left] at (C) {$O$};

   \draw[black, densely dotted] (C) -- +(0.5,0.5);
\end{tikzpicture}
\caption{$\frac d2\leq \alpha$}
\label{fig2}
\end{minipage}
\end{figure}

In order to state our result,  for $\alpha\in (0,1]$,  let  $\mathcal Q$ denote the closed quadrangle with vertices 
\Be
\label{aa}
		O:=(0,0), \quad \Big(\frac12, \frac12 \Big),  \quad A:= \Big(\frac12, 0\vee \frac{d-2\alpha}{2d}\Big), \quad A', 
	\Ee
where  we denote $P'=(y,x)$ for  a point $P= (x,y)\in [0,1]\times [0,1]$. 
In addition to  \eqref{scaling}, we also assume $\gamma\ge 2$.  In general, it is not possible to obtain \eqref{f-inhom} without extra regularity when $\gamma <2$ (see Remark  \ref{rem:inhomo}). 
We also let, for $\alpha \in (0, d/\gamma)\cap (0,1]$,\footnote{We impose this condition for simpler presentation. However, it is still possible obtain the estimates even when  $\alpha \ge d/\gamma$.}
\Be
\label{cd}
 \begin{aligned}
	\qquad \quad  \ \  C &= \Big(\frac{d-\gamma \alpha}{2(d-\alpha)} ,\, \frac{(d-2\alpha)(d-\gamma\alpha)}{2d(d-\alpha)} \Big).
\end{aligned}
\Ee
Let $\mathcal H$ be convex hull of $(1/2, 1/2)$, $A, C, C',$ and $A'$. 
Then, the following gives the inhomogeneous estimate \eqref{f-inhom}. 

\begin{thm}\label{thm_inhom1}  
Let $\gamma\ge 2$, $\alpha\in (0, d/\gamma)\cap(0,1]$,  and $2 \le \tilde r, r<\infty$.  
	Suppose that    
	\[
	(1/\tilde r, 1/ r)\in \mathcal H \setminus\{C, C'\}
	\] and  the exponents $q, \tilde q\in [2, \infty]$ satisfy 
		\begin{align}
		0&\le \frac 1q \le \frac 1 {\tilde q\,'} \le  1 - \frac{d}{2\alpha} \Big( \frac1r  - \frac1{\tilde r}\Big),  \hspace{-2.5cm} &&  \text{ if } \  \frac1r \ge \frac 1{\tilde r}, \label{ran_q_up}\\
		&\,\, \frac{d}{2\alpha} \Big(\frac1{\tilde r}- \frac1r \Big) \le \frac 1q \le \frac 1 {\tilde q\,'} \le 1,   \hspace{-2.5cm}  && \text{ if } \  \frac1r \le \frac 1{\tilde r}.\label{ran_q_low}
	\end{align}
 Then, 
	we have the estimate \eqref{f-inhom} and 
	\begin{align}
	\label{inho-est}
		\Big\| \int_{[0,t]} \mathrm{e}^{-i(t - s)(-\Delta)^{\gamma/2}} F(s)\,\mathrm{d}\mu (s) \Big\|_{L_t^q\left( \mathrm{d}\mu; L_x^{r} \right)} 
		\le C  \| F\|_{L_s^{\tilde{q}'}(\mathrm{d}\mu ; L_x^{\tilde{r}' })}.
	\end{align}
\end{thm}

\begin{comment}
The condition  $\lambda_{\alpha}(\tilde{r}, r, \tilde{q}, q) \ge0$ determines the range of $(\tilde r, r)$. 
Indeed,  when $\alpha<\alpha$,   if $(1/\tilde r, 1/r) \in \mathcal Q$ and  $\lambda_{\alpha}(\tilde{r}, r, \tilde{q}, q) \ge0$ holds for  some $\tilde q, q$ satisfying  \eqref{ran_q_up} and \eqref{ran_q_low},\footnote{Indeed, this is equivalent to $ \frac{d}{2} (1 - \frac1r - \frac1{\widetilde r} +\frac{\alpha-\alpha}{\alpha} | \frac1r - \frac1{\widetilde r}|) - \alpha\ge 0$.}  then $(1/\tilde r, 1/r)$ is contained in  $\mathcal H$.    
A similar estimate also possible for the wave operator, i.e., $\gamma=1$ (see Remark \ref{inhomo-wave}). 
\end{comment}

When $\gamma=2$ and ${\alpha}  =1$,  the result in Theorem \ref{thm_inhom1} coincides with those in \cite{Fos, Vil}.  If we allow additional regularity parameter  as the estimate \eqref{rem:inhomo}, it is possible to extend the range of admissible exponents (see Remark \ref{rem:inhomo}).

We now provide some words on the range of $\tilde{r}, r, \tilde q,$ and $q$ for the estimates  \eqref{f-inhom} and \eqref{inho-est}, which  varies depending on the associated parameters. 
If the point $(1/\tilde{r}, 1/r)$ lies in  $\mathcal{H}$, which  is the convex hull of  the points $A$, $A'$, $C$, $C'$, and $(1/2, 1/2)$ as depicted in Figure~\ref{fig_lambda} and Figure \ref{fig2}, then for each $(\tilde{r}, r) \in \mathcal{H}$, there exists a pair $(\tilde{q}, q)$ such that $\sigma_{\alpha}^\gamma(\tilde{r}, r, \tilde{q}, q)=0$.
In the case $\alpha\ge d/2$, which occurs only for $d=1, 2$, the quadrangle $\mathcal Q$ becomes the  square $[0,1/2]\times [0,1/2]$ (see Figure~\ref{fig2}), and Theorem~\ref{thm_inhom1} holds for $(1/\widetilde r, 1/r)$ in the square. 
When  $\alpha<d/2$, $\mathcal Q$ is a quadrangle properly contained in $[0,1/2]\times [0,1/2]$ (see Figure \ref{fig_lambda}).

When $\alpha = 1$, the region $\mathcal{H}$ coincides with the pentagon described in \cite{Vil}.
In Figure~\ref{fig2} (i.e., when $\alpha \ge d/2$), the points $A$ and $A'$ in Figure~\ref{fig_lambda} become $(1/2, 0)$ and $(0, 1/2)$, respectively. In this case, the admissible set of $(1/\tilde{r}, 1/r)$ is the square with vertices $O$, $(0, 1/2)$, $(1/2, 0)$, and $(1/2, 1/2)$, excluding the closed  boundary line segments connecting $O$ to $(0, 1/2)$ and $O$ to $(1/2, 0)$.

\subsection{$L^2$ local smoothing estimates} 
Unlike the wave operator, the fractional Schrödinger operator for $\gamma > 1$ is known to exhibit a smoothing property when integrated over a bounded space-time region. This phenomenon is referred to as {\it local smoothing}, and it plays an important role in the study of nonlinear dispersive equations (see, for example, \cite{KPV1, KPV2}). In particular, the local smoothing effect for the KdV equation goes back to the work of Kato \cite{Kato}. The sharp $L^2$ local smoothing estimate for the Schrödinger operator was independently obtained by Sjölin \cite{Sjo} and Vega \cite{Vega1, Vega2}. Local smoothing for more general dispersive equations was studied by Constantin and Saut \cite{CS}.

\newcommand{\R}{\mathbb{R}}
\newcommand{\B}{\mathrm{B}}
\newcommand{\dimA}{\dim_{\mathrm{A}}}
\newcommand{\Ncov}{N}

Let  $\mathbb{B}^d$ denote a unit ball in $\mathbb R^d$.  We now consider local smoothing estimates for $U_t^\gamma$ in the similar settings discussed above.  
 We have the following smoothing estimate relative to a fractal measure $\rd \mu(t)$.

\begin{thm}\label{thm_LS}
Let $\gamma >1$.  
Suppose that   $\mu$ is a measure  $\mu$ with $\supp \mu\subset [0,1]$ satisfying  \eqref{alpha-dim}.
Then, for any $s \geq  \alpha 
(1-\gamma)/2$, there exists $C$ such that
\Be 
\label{ls-s}
	\left\| U_t^\gamma f \right\|_{L_{t}^2 (\rd\mu; L^2_x(\mathbb{B}^d))} \leq C \|f\|_{H^s}.
\Ee
\end{thm}

We show that the regularity exponent $\alpha(1 - \gamma)/2$ is sharp, in the sense that there exist a set $E$ and a measure $\mu$ supported on $E$ satisfying the assumptions of Theorem~\ref{thm_LS}, while the estimate \eqref{ls-s} fails unless $s \ge\alpha(1-\gamma)/2$ (see Proposition~\ref{prop:nec2}).
Our result can also be extended to general symbols $P(\xi)$ in place of $|\xi|^\gamma$; see \eqref{250629_1530} and \eqref{250629_1531}.
 
 We prove  Theorem \ref{thm_LS} using a temporal localization lemma, Lemma \ref{lem_CLV} in Section~\ref{sec:l2-ls}. By this lemma, it is possible to obtain estimates in a similar form as \eqref{f-stri} under an assumption weaker than that of bounded Assouad characteristic.  
 To this end, let
$ E \subset [0,1] . $
 
 We now recall the notion of  Assouad spectrum  \cite{FraYu}.  For $\theta\in (0,1)$,   the \emph{Assouad spectrum} of $E$ at parameter $\theta$ is given by 
\[
  \dimA^{\theta} E
  \;=\;
  \inf \Bigl\{
    \alpha>0 \;:\; \exists c>0 \text{ s.t. } 
    \forall\, I, 
    \Ncov\!\bigl(E\cap I, |I|^{1/\theta}\bigr)
    \le
    c \Bigl(\frac{|I|}{|I|^{1/\theta}}\Bigr)^{\alpha}
  \Bigr\},
\]
where $I$ runs over any interval in $[0,1]$.   
Note that $
\dim_{\mathrm M}(E)=\lim_{\theta\to 0^+}\dimA^\theta(E)$ and $ 
\dimA(E)=\lim_{\theta\to 1^-}\dimA^\theta(E).$
The Assouad spectrum interpolates between the Minkowski and Assouad dimensions. More precisely, a general comparison principle \cite[Proposition 3.1]{FraYu} states that
\[
\dim_{\mathrm M}(E)\le \dimA^\theta(E)\le 
\min\Bigl\{\frac{\dim_{\mathrm M}(E)}{1-\theta},\,\dimA(E)\Bigr\}.
\]
For instance, for the set \(E_a=\{n^{-a}:n=1,2,3,\dots\}\),  
$
\dimA^\theta(E_a)=\frac{1}{(1+a)(1-\theta)}\wedge 1,
$
and hence \(\dimA^\theta(E_a)=1\) whenever \(\theta\ge \tfrac{a}{1+a}\).

In analogy with the bounded Assouad characteristic,  we define 
\[  [E]_{\alpha, \theta}   =\sup_{I}    \Bigl(\frac{|I|^{1/\theta}}{|I|}\Bigr)^{\alpha}   \Ncov\!\bigl(E\cap I, |I|^{1/\theta}\bigr)
    .\]
   Clearly,  $[E]_{\alpha, \theta}\le [E]_\alpha$. 
    We say $E$ has bounded $(\alpha, \theta)$-Assouad characteristic if $[E]_{\alpha, \theta}<\infty$. 
    Then we have the following theorem.

\begin{thm}\label{thm_LS-homo}
Let $\gamma >1$, and $E\subset [0,1]$.  
Suppose that $E$ has bounded $(\alpha, (\gamma-1)/\gamma)$-Assouad characteristic.
Then, for any $s \geq  
(\alpha-\gamma)/2$, there exists $C$ such that
\Be 
\label{ls-s-homo}
	\sup_{j\ge 1} \left\| U_t^\gamma P_j f \right\|_{L_{t}^2 ((E(2^{-\gamma j}); L^2_x(\mathbb{B}^d))}  +  \big\| U^\gamma_t P_{\le 0} f\big\|_{L_t^2( E(1); L_x^2(\mathbb{B}^d))} \leq C \|f\|_{H^s}.
\Ee
\end{thm}
The regularity exponent $(\alpha - \gamma)/2$ is sharp, in the same sense of the regularity exponent of Theorem~\ref{thm_LS} (see Proposition~\ref{prop:nec2}).

\subsubsection* {Organization}  In Section \ref{sec_hom}, we are mainly concerned with the homogeneous estimates. The inhomogeneous estimates are shown 
in   Section \ref{sec:inho}. The local smoothing estimates are discussed in Section \ref{sec:l2-ls}.  

\vspace{-5pt}

\subsubsection*{Notation}  Throughout the paper, by $C$ we denote constants which may vary from line to line.  We also use $A \lesssim B$ to mean $A \le CB$ for harmless constants $C$.

\section{Homogeneous estimates}\label{sec_hom}
 
In this section, we prove Theorem \ref{thm_main} and Theorem \ref{cor_main}.  
We begin by recalling  \eqref{l-p} and $P_{\le 0}= {\rm Id} - \sum_{j>0} P_j$.

\subsection{Proof of Theorem \ref{thm_main}}  The following proposition contains the key estimates for the proofs of Theorem \ref{thm_main} and \ref{cor_main}. 

\begin{prop}\label{lem_Str} Let  $\alpha\in (0,1]$ and $\gamma\in \mathbb R_{>0} \setminus \{1\}$, and  $2\le q,r \leq \infty$ with $(d/\alpha, q,r)\neq (2 ,2,\infty)$.  Suppose that $E$ has bounded $\alpha$-Assouad characteristic. Then, for $j\in \mathbb Z$,  we have
	\Be
	\label{250309_0149c} 
		 \| U_{t}^\gamma P_j f\|_{L_{t}^q(E(2^{-\gamma j})  ;L^r_x)} \leq C 2^{(\frac d2 - \frac dr-\frac \gamma q)j}  [E]_\alpha^\frac1q \|f\|_{2}  
	\Ee
	whenever \eqref{con-ad}  holds.
\end{prop}

 Let $M$ be an invertible affine map on $\mathbb R$, i.e., $M(t)=at+b$ for some $a\neq 0$ and $b\in \mathbb R$. 
Then, from the definition   \eqref{assouad} one can easily see that 
\Be 
 [M(E)]_\alpha=[E]_\alpha.
\Ee 
Using this and scaling argument, one can also deduce  the estimate  \eqref{250309_0149c} for any $j$ from that for $j=0$. 
However, we provide a direct proof keeping track of the bounds in terms of $j$.

Let $\tilde\psi\in C_c^\infty(1/2, 2)$  such that $\tilde\psi=1$ on $\supp \psi$. Let $\tilde P_j f$ be defined by ${\tilde P_j f}=(\tilde \psi(|\xi|/2^{j}) \widehat f\,)^\vee$.

\begin{proof}[Proof of Theorems \ref{thm_main} and \ref{homo-homo}] 
The estimate \eqref{250309_0149c} implies
	\[	 \| U_{t}^\gamma P_j f\|_{L_{t}^q(E(2^{-\gamma j})  ;L^r_x)} \leq C 2^{(\frac d2 - \frac dr-\frac \gamma q)j}  [E]_\alpha^\frac1q \|\tilde P_jf\|_{2}  .\]
	Since $ 2^{(\frac d2 - \frac dr-\frac \gamma q)j}  [E]_\alpha^\frac1q \|\tilde P_j f\|_{2} \le C\|f\|_{\dot H^{s_\gamma (q,r)}}$ for all $j\in \mathbb Z$, 
this proves Theorem  \ref{homo-homo} via Proposition \ref{lem_Str}. 

Since $2^{(\frac d2 - \frac dr-\frac \gamma q)j}  [E]_\alpha^\frac1q \|\tilde P_j f\|_{2} \le C\|f\|_{H^{s_\gamma (q,r)}}$ for $j\ge 1$, 
 thanks to Proposition \ref{lem_Str}, we only need to consider $U_t^\gamma P_{\le 0} f$.  
Using  Berstein's inequality and Plancherel's theorem, we have 
\Be 
\|U_t^\gamma P_{\le 0} f \|_{L_x^r}\le C\|P_{\le 0} f \|_{L_x^2}
\Ee
for  $r\ge 2$.  Since $E$ is a bounded set, we have $|E(1)|\le (2+\diam E)[E]_\alpha$. Thus,  it follows that 
\[ \| U_t^\gamma P_{\le 0} f\|_{L_t^q(E(1); L_x^r)} \le C[E]_\alpha^\frac1q\|P_{\le 0} f \|_{L_x^2} \le C\|f\|_{H^{s_\gamma (q,r)}}. \qedhere\]
\end{proof}

To prove Proposition \ref{lem_Str}, we discretize the estimate \eqref{250309_0149c}. To this end, 
 let $\mathcal E_j= \{ \tau \}$  denote   a maximally $2^{-\gamma j }$-separated subset of   $E$ as in the introduction.  
For each $\tau\in  \mathcal E_j$, we set 
\[ I_{\tau}= (\tau - 2^{1-\gamma j }, \tau+ 2^{1-\gamma j }).\]  
Then, it is clear  that
\begin{align}
\label{inE}
E\subset 
	  E(2^{-\gamma j}) \subset \bigcup_{\tau \in  \mathcal E_j } I_{\tau}.
\end{align}

 To discretize the estimate, we make use of  the following lemma, which is based on a locally constant property of $U_t^\gamma P_j$.

 \begin{lem}\label{lem_250901}  Let $1\le q,r\le \infty$, and $\tau\in \mathbb R$.  Then,  for  $j\in \mathbb Z$, we have
 \Be 
 \label{equiv}
 \| U_t^\gamma P_j f\|_{L_t^q(I_{\tau}; L_x^r)}\sim 2^{-j \frac \gamma q} \|U_{\tau}^\gamma P_j f\|_{L_x^r},
 \Ee
 with the implicit constants independent of $j$. 
 \end{lem} 
 
 \begin{proof}  Let $t\in  I_\tau$. 
 After scaling, we write 
\begin{align*}
	U_t^\gamma P_j f(x) 
	= \int_{\mathbb{R}^d} \e^{i2^{\gamma j }(t-\tau)|\xi|^\gamma} \e^{i(2^{\gamma j }\tau|\xi|^\gamma+2^j\langle x, \xi\rangle)} \psi(|\xi|)2^{dj}\widehat{f}(2^j \xi)\,\mathrm{d}\xi.
\end{align*}
Expanding $\e^{i2^{\gamma j }(t-\tau)|\xi|^\gamma} \psi(|\xi|)$ in a Fourier series on $[-\pi, \pi)^d$, we have  
\[
	\e^{i2^{\gamma j }(t-\tau)|\xi|^\gamma} \psi(|\xi|) = \sum_{\ell\in\mathbb{Z}^d} C_\ell \e^{i \langle\ell,\xi\rangle},
\]
where $C_\ell = \int \e^{i(2^{\gamma j }(t-\tau)|\xi|^\gamma - \langle\ell, \xi\rangle)} \psi(|\xi|) \,\mathrm{d}\xi$.  Since $|t-\tau|\le 2^{1-\gamma j}$, by routine integration by parts,   it is easy to see
\[
	|C_\ell| \leq \mathbf c_\ell:=C_N (1+|\ell|)^{-N}, \quad \forall N>0,
\]
with $C_N$  independent of $j$ and $t\in  I_{\tau} $. Thus, for $t\in I_{\tau} $, we have
\begin{align}\label{250309_0134}
	|U_t^\gamma P_j f(x)| \le  \sum_{\ell\in\mathbb{Z}^d} \mathbf c_\ell |U_{\tau}^\gamma P_j f(x+2^{-j}\ell)|.
\end{align}
Since $|I_{\tau}|\sim 2^{-\gamma j}$,  using \eqref{250309_0134} and Minkowski's inequality,  we obtain
\[\| U_t^\gamma P_j f\|_{L_t^q(I_{\tau}; L_x^r)}\le  C 2^{-j \frac \gamma q} \|U_{\tau}^\gamma P_j f\|_{L_x^r}\]
for a constant $C$. 

To prove the reverse inequality, we note that the roles of $t$ and $\tau$  \eqref{250309_0134} can be interchanged.  
Then, the rest of argument is identical to that for the opposite  inequality. 
 \end{proof}

From \eqref{inE},   it follows that
\begin{align}\label{250308_1600}
	\|U_t^\gamma P_j f\|_{L_t^q(E(2^{-\gamma j});  L_x^r)} \leq \Big( \sum_{\tau \in  \mathcal E_j } \| U_t^\gamma P_j f\|_{L_t^q(I_{\tau}; L_x^r)}^q \Big)^{1/q}.
\end{align}
Combining \eqref{equiv} and  \eqref{250308_1600} yields 
\Be 
\label{discrete}
\|U_t^\gamma P_j f\|_{L_t^q(E(2^{-\gamma j});  L_x^r)} \leq C 2^{- j \frac \gamma q}   \| U_{\tau}^\gamma P_j f\|_{\ell_{\tau}^q( \mathcal E_j; L^r_x)}.
\Ee
 Therefore,   the desired estimate \eqref{f-stri} in Theorem \ref{thm_main}  follows once  the following proposition is established. 

\begin{prop}\label{lem_discrete_Str}   Let $\gamma\in \mathbb R_{>0} \setminus \{1\}$. Let  $\alpha\in (0,1]$ and  $2\le q,r \leq \infty$ with $(d/\alpha, q,r)\neq (2 ,2,\infty)$. Suppose that $E$ has bounded $\alpha$-Assouad characteristic and $\mathcal E_j$  is  a $2^{-\gamma j }$-separated subset of   $E$. Then we have
	\Be
	\label{250309_0149} 
		 \| U_{\tau}^\gamma P_j f\|_{\ell_{\tau}^q( \mathcal E_j  ;L^r_x)} \leq C 2^{jd(\frac 12 - \frac 1r)} [E]_\alpha^\frac1q \|f\|_{2}
	\Ee
	with $C$ independent of  the choice of  $\mathcal E_j$ whenever \eqref{con-ad}  holds.
\end{prop}

The estimate \eqref{250309_0149} may be viewed as a discrete analogue of a Strichartz estimate. Its proof relies on a combination of the standard $TT^*$ argument and the notion of $\alpha$-bounded characteristic. The role of the dimensionality index $\alpha$ in determining the admissible range \eqref{con-reg} will become evident in the course of the proof.

\subsection{Proof of Proposition~\ref{lem_discrete_Str}}  
\label{sec:2-3}
When $r=2$,  \eqref{con-ad}  implies $q=\infty$. In this case, the estimate 
\eqref{250309_0149} follows from Plancherel's theorem. Thus, we need only to consider
$r>2$.

Let us set
\begin{align}\label{r*} r_\ast=\frac{2d}{d-2\alpha}.\end{align}
Recalling the condition \eqref{con-ad}, we distinguish the admissible pairs $(q,r)$ into  two cases, \emph{the non-endpoint $(q, r)\neq (2, r_\ast)$} and  
\emph{the endpoint case $(q,r)=(2, r_\ast)$}.  
Since we are excluding the case $(d/\alpha, q,r)=(2, 2, \infty)$, the endpoint case is  only relevant when  
\[   d >  2 \alpha .\]  
This  is always satisfied  when $d\ge 3$ since $\alpha\in (0,1]$. When $d=2$,  we need to distinguish the cases $\alpha\in (0,1)$ and $\alpha=1$ but  we have $r_\ast=\infty$ in the latter case, which is already excluded.  Also, 
when $d=1$,   we need to distinguish  the cases  $\alpha<1/2$ and   $\alpha\in [ 1/2,1]$.  Nonetheless, for the cases $(d=1, \alpha \in [1/2, 1])$ and $(d=2, \alpha = 1)$, all the desired estimates will be established in the non-endpoint regime. (See Section~\ref{non-end} .)

For the non-endpoint case, we use the standard $TT^*$ argument, while the endpoint case is treated  using the bilinear interpolation method from \cite{KeelTao}.

\subsubsection{Proof of Proposition~\ref{lem_discrete_Str}: Non-endpoint case}
\label{non-end}
Note that the adjoint operator $(U_{\tau}^\gamma P_j)^*$ of $U_{\tau}^\gamma P_j$ is given by 
\begin{align*}
	(U_{t}^\gamma P_j)^* F(\tau) = &(2\pi)^{-d} \int_{\mathbb{R}^d} \e^{i(-t|\xi|^\gamma + \langle x, \xi\rangle)} \psi_j(|\xi|) \widehat {F(\tau)}(\xi)\,\mathrm{d}\xi,
\end{align*}
where $F(\tau)=F(\cdot, \tau)$. 
By duality,  the estimate \eqref{250309_0149}
is equivalent to  
\[ \|\sum_{\tau}  (U_{\tau}^\gamma P_j)^* F(\tau)\|_{L^2_x} \le C 2^{jd(\frac 12 - \frac 1r)} [E]_\alpha^\frac1q \|F\|_{\ell_\tau^{q'}(\mathcal E_j; L^{r'}_x)}, \]
which we can equivalently write as follows: 
\begin{align}
\label{hoho}
	\Big|  \sum_{\tau\in \mathcal E_j}  \Big\langle  \sum_{ \tau'\in  \mathcal E_j }  T_{\tau, \tau'} F(\tau), F(\tau')\Big\rangle\Big|
	\leq &C 2^{2jd(\frac 12 - \frac 1r)} [E]_\alpha^\frac2q \|F\|_{\ell_\tau^{q'}(\mathcal E_j; L^{r'}_x)}^2, 
\end{align}
where
\Be
\label{utut} 
T_{t, s} g = U_{t}^\gamma P_j (U_{s}^\gamma P_j)^* g.\Ee

Thus, by H\"older's inequality  the estimate \eqref{250309_0149} follows if we  show
\begin{align}
\label{ineq_TT*}
\|  \sum_{\tau' \in  \mathcal E_j } T_{\tau, \tau'}  F (\tau')  \|_{\ell_\tau^{q}(\mathcal E_j; L^{r}_x)}
	\leq &C 2^{2jd(\frac 12 - \frac 1r)}[E]_\alpha^\frac2q \|F\|_{\ell_\tau^{q'}(\mathcal E_j; L^{r'}_x)}. 
\end{align}

\begin{lem}\label{rr'} For $t,s\in \mathbb R$, let 
\Be 
\label{gr}
 \mathcal K_r(t, s) =(1+ 2^{\gamma j }|t-s|)^{-d(\frac1 2 - \frac1 r)} .
 \Ee
Then, for $2\le r\le \infty$ and $t,s\in \mathbb R$, we have 
\Be\label{250309_1101}
	\| T_{t, s} g\|_{L_x^r} \le C 2^{2jd(\frac1 2 - \frac1 r)} \mathcal K_r(t, s)\| g\|_{L_x^{r'}}. 
\Ee
\end{lem} 

\begin{proof} 
The Plancherel theorem gives
\Be\label{250309_1459}
	\| T_{t, s}  g \|_{L_x^2} \leq \| g\|_{L^2_x}
\Ee
for $t,s\in \mathbb R$. Note that the kernel $K_{t, s}$ of  the operator $T_{t, s}$ is given by 
\[  K_{t, s}(x, y)=  (2\pi)^{-d} \int_{\mathbb{R}^d} \e^{i ( (t-s)|\xi|^\gamma + \langle x-y, \xi \rangle)}\psi^2(|\xi|/2^j)\,\mathrm{d}\xi.\]
Since the Hessian matrix of $|\xi|^\gamma$, $\gamma\neq 1$ has non-vanishing determinant when $|\xi|\sim 1$, 
after scaling $\xi\to 2^j\xi$, using  the stationary phase estimate, we have 
\begin{align*}
	\big|  K_{t, s}(x, y)\big| \le C 2^{dj} \min\{ ( 2^{\gamma j }|t - s|)^{-\frac d2}, 1\}, 
\end{align*}
which gives 
\Be\label{250309_1100}
	\| T_{t, s} g \|_{L_x^\infty} \le C 2^{dj}  (1+ 2^{\gamma j }|t - s|)^{-\frac d2} \|g\|_{L_x^1}.
\Ee
Interpolating \eqref{250309_1459} and \eqref{250309_1100} yields  \eqref{250309_1101} for $2\le r\le \infty$ and $t,s\in \mathbb R$.
 \end{proof}

By Minkowski's inequality and \eqref{250309_1101}, the left hand side of \eqref{ineq_TT*} is bounded by 
\begin{align*}
 C 2^{2jd(\frac1 2 - \frac1 r)}\Big\|  \sum_{\tau' \in  \mathcal E_j }  \mathcal K_r(\tau, \tau') \| F(\tau')\|_{L_x^{r'}}  \Big\|_{\ell_\tau^{q}(\mathcal E_j)}.
\end{align*}
Thus,  \eqref{ineq_TT*}  follows if we show \begin{align}\label{250309_1102}
\begin{split}
	\Big\|  \sum_{\tau' \in  \mathcal E_j }  \mathcal K_r(\tau, \tau') \| F(\tau')\|_{L_x^{r'}}  \Big\|_{\ell_\tau^{q}(\mathcal E_j)} 
	\le C [E]_\alpha^\frac2q\| F(\tau)\|_{\ell_\tau^{q'} (\mathcal E_j; L_x^{r'})}.
\end{split}
\end{align}

To prove this estimate,  we use the next lemma, which is a discrete variant of Young's convolution inequality. 

\begin{lem}\label{young}  Let $\mathcal E$ be a discrete subset of $\mathbb R$. 
\begin{itemize}
\item[{\bf (A)}] Suppose  that 
\[  
\sup_{\tau'\in \mathcal E} \|\mathcal K(\cdot, \tau')\|_{\ell^{s}_\tau (\mathcal E)}, \quad  \sup_{\tau \in \mathcal E}    \|\mathcal K(\tau,\cdot)\|_{\ell^{s}_{\tau'} (\mathcal E)} \le B\]  for some $1\le s\le \infty$. Then, we have 
\Be
\label{seq}
\Big\|  \sum_{\tau' \in  \mathcal E }  \mathcal K(\tau, \tau')\, a_{\tau'}  \Big\|_{\ell_\tau^{q}(\mathcal E)} \le C B \| a_\tau \|_{\ell_\tau^{p}(\mathcal E)} 
\Ee
if $1/p-1/q=1-1/s$ and $1\le p, q\le \infty$.  

\item[{\bf (B)}]  Suppose  that 
\[ \sup_{\tau'\in \mathcal E} \|\mathcal K(\cdot, \tau')\|_{\ell^{s,\infty}_\tau (\mathcal E)},\quad \sup_{\tau \in \mathcal E}    \|\mathcal K(\tau,\cdot)\|_{\ell^{s,\infty}_{\tau'} (\mathcal E)}\le B\] 
 for some $1< s< \infty$.  
Then,  we have the estimate  \eqref{seq} 
provided that $1/p-1/q=1-1/s$ and $1< p, q< \infty$.
\end{itemize}
\end{lem}

\begin{proof}[Proof of Lemma \ref{young}] We show {\bf (B)} first. Let us set 
\[ S(\{a_{\tau'}\})_\tau= \sum_{\tau' \in  \mathcal E }  \mathcal K(\tau, \tau')\, a_{\tau'}. \]
Then, by H\"older's inequality in Lorentz space we have 
\[ \|S(\{a_{\tau'}\})_\tau\|_{\ell_\tau^\infty} \le \sup_{\tau \in \mathcal E}    \|\mathcal K(\tau,\cdot)\|_{\ell^{s,\infty}_{\tau'} (\mathcal E)} \| a_{\tau'}\|_{\ell^{s',1}_{\tau'}}. \]
On the other hand, Minkowski's inequality gives
\[ \|S(\{a_{\tau'}\})_\tau\|_{\ell_\tau^{s,\infty}} \le \sup_{\tau'\in \mathcal E} \|\mathcal K(\cdot, \tau')\|_{\ell^{s,\infty}_\tau (\mathcal E)} \| a_{\tau'}\|_{\ell^{1}_{\tau'}}. \]
Real interpolation between those two estimates for the operator $S$ yields the estimate \eqref{seq}  whenever $1/p-1/q=1-1/s$  and $1< p, q< \infty$.  

The proof of {\bf (A)} is similar. Replacing the Lorentz space norms with $\ell^p$ norms,  complex interpolation gives  \eqref{seq} if $1/p-1/q=1-1/s$  and $1\le p, q\le \infty$. 
\end{proof}

We now recall that  $\mathcal E_j$ is a maximally   $2^{-\gamma j}$ separated subset of  $E$, which has bounded $\alpha$-Assouad characteristic;
this assumption will be used, sometimes tacitly, in the rest of the section. 
To prove Proposition~\ref{lem_discrete_Str} for the non-endpoint case, we need the following lemma. 

\begin{lem} \label{kernel} Let $r>2$, and let  $\mathcal K_r$  be given by \eqref{gr}. If  $d(\frac1 2 - \frac1 r)> \frac \alpha s $, then we have
\Be 
\label{strong}
\sup_{\tau'\in \mathcal E_j}  \|\mathcal K_r(\cdot, \tau')\|_{\ell^{s}_\tau (\mathcal E_j)}, \quad  \sup_{\tau \in \mathcal E_j}    \|\mathcal K_r(\tau,\cdot)\|_{\ell^{s}_{\tau'} (\mathcal E_j)} \le C[E]_\alpha^\frac1s
\Ee
for a constant $C$, independent of $j$.   Moreover,  if $d(\frac1 2 - \frac1 r)= \frac \alpha s $, then 
\Be \label{weak}
\sup_{\tau'\in \mathcal E_j}  \|\mathcal K_r(\cdot, \tau')\|_{\ell^{s,\infty}_\tau (\mathcal E_j)}, \quad  \sup_{\tau \in \mathcal E_j}    \|\mathcal K_r(\tau,\cdot)\|_{\ell^{s,\infty}_{\tau'} (\mathcal E_j)} \le C[E]_\alpha^\frac1s
\Ee for a constant $C$, independent of $j$.  
\end{lem}

For $k\ge 0$, we set  
\Be
\label{chi-k}
\chi_k (t)= 
\begin{cases} 
	 \chi_{\{ u \in\mathbb R: 2^{k-1}< 2^{\gamma j } |u|\le  2^k\}}(t),  &\quad k\ge 1,
	  \\[2pt]
	\   \   \ \chi_{\{ u \in\mathbb R:  2^{\gamma j } |u|\le  1\}}(t) ,  &\quad k=0, 
	\end{cases} 
\Ee
and  
\Be 
\label{K-chi}
 \mathcal K_{r,k}(t,s)= \mathcal{K}_{r} (t,s) \chi_{k}(t-s).\Ee 
 Thus, it follows that
 \Be 
\label{SK-chi}
\mathcal K_{r}(t,s)= \sum_{k\ge 0} \mathcal{K}_{r, k} (t,s).
\Ee

\begin{proof}[Proof of Lemma~\ref{kernel}]  To prove the first assertion \eqref{strong}, by symmetry  it suffices to show the first inequality in \eqref{strong}. 
Recalling \eqref{chi-k} and \eqref{gr},  by \eqref{SK-chi} it follows that
\begin{align*} 
 \|\mathcal K_r(\cdot, \tau')\|_{\ell^s(\mathcal E_j)}^s=  \sum_{k\ge 0} \| \mathcal K_{r,k}(\cdot, \tau')\|_{\ell^s(\mathcal E_j)}^s  \leq C \sum_{k\geq0} 2^{-k d(\frac1 2 - \frac1 r)s} \| \chi_{k}(\cdot - \tau')\|_{\ell^s(\mathcal E_j)}^s . 
 \end{align*} 
 Note that $\| \chi_{k}(\cdot - \tau')\|_{\ell^s(\mathcal E_j)}^s= \# A_k (\tau')$ where 
\begin{equation}
\label{akk}
	A_k (\tau') =  
	\begin{cases} \  \{\tau \in  \mathcal E_j  : 2^{k-1} < 2^{\gamma j}|\tau-\tau'| \leq 2^{k }\},   \quad  & k\ge 1,
	\\[4pt]
		 \   \{\tau \in  \mathcal E_j  : |\tau-\tau'| \leq 2^{- \gamma j }\},   \quad & k=0.
	\end{cases} 
	\end{equation} 
	  Since $\mathcal E_j$ is a maximally   $2^{-\gamma j}$ separated subset of  $E$ with bounded $\alpha$-Assouad characteristic, from  \eqref{assouad} it follows 
\Be\label{ak}  \# A_k(\tau')\le C2^{k\alpha}[E]_\alpha, \quad k\ge 0.\Ee Thus,   we obtain
\[ \|\mathcal K_r(\cdot, \tau')\|_{\ell^s(\mathcal E_j)}^s  \le C \sum_{k\geq0} 2^{-k d(\frac1 2 - \frac1 r)s}  2^{k\alpha} [E]_\alpha.\]
 The sum  is bounded by  $C[E]_\alpha$ since  $d(\frac 1 2 - \frac 1 r)s - \alpha >0$. This proves the first inequality in \eqref{strong}. 

 To show   \eqref{weak}, as before,  we need only  to show the first inequality, which  is equivalent to  the estimate  
\[  \# \{  \tau\in \mathcal E_j:  \mathcal K_r(\tau, \tau')>\lambda \} \le  C \lambda^{-s} [E]_\alpha  \] 
with a constant $C$ independent of $\lambda$ and $\tau'$.  Recall \eqref{gr}. Since  $\mathcal K_r\le 1$, it suffices to consider $0<\lambda\le 1$.  Note that  the set 
$\{  \tau\in \mathcal E_j:  \mathcal K_r(\tau, \tau')>\lambda \}$ is contained in an interval of length $C 2^{-\gamma j}  \lambda^{- 1/d(\frac1 2 - \frac1 r)}$.  
Additionally, $\mathcal E_j$ is a maximally   $2^{-\gamma j}$ separated subset of  $E$ with bounded $\alpha$-Assouad characteristic. Since $d(\frac1 2 - \frac1 r)= \frac \alpha s $,  from \eqref{assouad}  
the desired inequality  follows ({\it cf.}, \eqref{ak}). 
\end{proof}

Combining  Lemmas \ref{young}   and   \ref{kernel}, we prove  the estimate 
\eqref{250309_1102} for the non-endpoint case. 

\begin{proof}[Proof of \eqref{250309_1102}]
Note that ${\bf (A)}$ in  Lemma \ref{young}  and \eqref{strong} give
\[\Big\|  \sum_{\tau' \in  \mathcal E_j }  \mathcal K_r(\tau, \tau') \| F(\tau')\|_{L_x^{r'}}  \Big\|_{\ell_\tau^{q}(\mathcal E_j)} 
	\le C [E]_\alpha^\frac1s\| F(\tau)\|_{\ell_\tau^{q'} (\mathcal E_j; L_x^{r'})}.\]
for  $(q,r)\in [2,\infty]\times [2,\infty]$ satisfying  $1/{q'}-1/q=1-1/s$ and  \eqref{con-ad} with strict inequality, i.e., $\frac d2 (\frac 12 - \frac1r) >  \frac \alpha q$.
Thus,  the estimate  
\eqref{250309_1102}  follows. Similarly, by ${\bf (B)}$  in  Lemma \ref{young}  and \eqref{weak}  we obtain 
\eqref{250309_1102} for  $(q,r)\in [2,\infty]\times [2,\infty]$ satisfying $\frac d2 (\frac 12 - \frac1r) =  \frac \alpha q$ and $q\neq 2$.
\end{proof}

Therefore,  we have established the estimate \eqref{250309_0149}  for the non-endpoint case 

\subsubsection{Proof of Proposition~\ref{lem_discrete_Str}: Endpoint case}\label{sec_endpoint}

To show   \eqref{250309_0149} for $ q=2$ and $r=r_\ast$, we make use of the bilinear interpolation argument in \cite{KeelTao}. 
To this end,  recalling \eqref{hoho}, we consider a bilinear form
\[
	\mathcal{B}(F,G) = \sum_{\tau, \tau'\in  \mathcal E_j }  \langle T_{\tau, \tau'}  F(\tau), G(\tau') \rangle_x.
\]
From \eqref{hoho} with $(q,r)=(2, r_\ast)$, we note that  \eqref{250309_0149} for $(q, r)=(2, r_\ast)$ follows from
\Be
\label{goal} 
|\mathcal{B}(F,G)|\le C  2^{2jd(\frac 12 - \frac 1{r_\ast})}  [E]_\alpha  \|F\|_{\ell_\tau^{2}(\mathcal E_j; L^{r'_\ast}_x)}\|G\|_{\ell_\tau^{2}(\mathcal E_j; L^{r'_\ast}_x)}.
\Ee

Recalling  \eqref{chi-k}, we decompose 
\[ \mathcal{B} (F,G)= 
\sum_{k=0}^\infty   \mathcal{B}_k (F,G), \]
where
\begin{align}
\label{sumk} 
	\mathcal{B}_k (F, G) =  \sum_{\tau, \tau'}  \chi_k (\tau - \tau') \langle T_{\tau, \tau'}  F(\tau), G(\tau') \rangle_x,  \quad k\ge 0.  
\end{align}
Using the  bounds in the previous section,  we obtain the following. 

\begin{lem}\label{lem_interpol1}
For $1\le a, b\le \infty$, let 
 \[ \lambda(a, b) = \frac d2 \Big(1 - \frac1a - \frac1b\Big).\] 
If  $(1/{a}, 1/{b})$ are contained in the interior of the quadrangle $\mathcal Q$ $($with vertices $(0,0)$,  $A=(1/2, 1/r_\ast)$,   $(1/2, 1/2)$, and  $A'=(1/r_\ast, 1/2),$ see \eqref{r*}$)$, 
  we have
	\Be
	\label{k-loc}		|\mathcal{B}_k (F, G)| \lesssim 2^{2\lambda(a,b)j } 2^{(\alpha-\lambda(a, b))k}  [E]_\alpha \|F \|_{\ell_\tau^2(\mathcal E_j; L_x^{a'})}\|G \|_{\ell_\tau^2(\mathcal E_j; L_x^{b'})}.
\Ee
\end{lem}

\begin{proof}
The estimate \eqref{k-loc} for $(a,b)=(2,2)$ and $(a,b)=(\infty, \infty)$  is easy to show. Indeed, from \eqref{sumk}  and \eqref{250309_1101}, we have
\[ |\mathcal{B}_k (F, G)|  \le C  2^{2j\lambda(r,r)}  \sum_{\tau, \tau'} \mathcal K_{r,k}(\tau, \tau')  \| F(\tau)\|_{L_x^{r'}}\| G(\tau')\|_{L_x^{r'}}\]
for $2\le r\le \infty$, where    $\mathcal K_{r,k}$ is given by  \eqref{K-chi}.

Therefore,  the Cauchy--Schwarz inequality gives
\Be 
\label{00-22}
 |\mathcal{B}_k (F, G)|  \le 2^{2j\lambda(r,r)}  \Big\|  \sum_{\tau} \mathcal K_{r,k}(\tau, \tau')\| F(\tau)\|_{L_x^{r'}}\Big\|_{\ell^2(\mathcal E_j)}  \|G(\tau')\|_{\ell^2(\mathcal E_j; L_x^{r'})}.
 \Ee
From \eqref{gr}, it follows that $\mathcal K_{r,k} \le  2^{-\lambda(r, r) k}\chi_{k}$.   Since $\| \chi_{k}(\cdot - \tau')\|_{\ell^1(\mathcal E_j)}= \# A_k (\tau')$,  recalling \eqref{ak},  we see that 
\Be \label{kr-norm} \sup_{\tau'\in \mathcal E_j}  \|\mathcal K_{r,k}(\cdot, \tau')\|_{\ell^{1}_\tau (\mathcal E_j)},  \,\,  \sup_{\tau \in \mathcal E_j}    \|\mathcal K_{r,k}(\tau,\cdot)\|_{\ell^{1}_{\tau'} (\mathcal E_j)}   \le  C  2^{(\alpha-\lambda(r, r))k} [E]_\alpha.\Ee
Therefore, applying  {\bf (A)} in  Lemma \ref{young} with $p=q=2$,  we have \eqref{k-loc}, in particular,  for  $(a,b)=(2,2)$ and $(a,b)=(\infty, \infty)$.

Therefore, in perspective  of interpolation, it suffices to show
	\begin{align}
	\label{1st}
		|\mathcal{B}_k (F, G)| 
	&\lesssim 2^{d(\frac12 - \frac1r)j } 2^{( \alpha-\frac d2 (\frac12 - \frac1r) )k} [E]_\alpha \|F\|_{\ell_\tau^2(\mathcal E_j; L_x^2)} \| G\|_{\ell_\tau^2(\mathcal E_j; L_x^{r'})},
	\\[4pt]
	\label{2nd}
	|\mathcal{B}_k (F, G)| 
	&\lesssim 2^{d(\frac12 - \frac1r)j } 2^{( \alpha-\frac d2 (\frac12 - \frac1r) )k}  [E]_\alpha \|F\|_{\ell_\tau^2(\mathcal E_j; L_x^{r'})} \| G\|_{\ell_\tau^2(\mathcal E_j; L_x^2)}
	\end{align}
	for $2\leq r < r_\ast$.  By symmetry, we only have to show \eqref{1st}. Indeed, observe that 
	\[ \mathcal{B}_k (F, G)= \overline{\mathcal{B}_k (G, F)} .\] 
	Thus, \eqref{2nd} follows from \eqref{1st}.

Thanks to the cutoff function $\chi_{k}(\tau - \tau')$, using a standard argument, we may assume that $\supp F(x,\cdot)$ and $\supp G(x,\cdot)$ are contained  in an interval of length $2^k2^{-\gamma j }$ for all $x$ (see Remark \ref{interval} below).

We now claim that
\Be\label{claim_endpoint}
	\Big\| \sum_{\tau'} \chi_{k}(\tau -  \tau') (U_{\tau'}^\gamma P_j)^* G(\tau') \Big\|_{L_x^2} \!
	\lesssim  2^{d(\frac12 - \frac1r)j} 2^{( \frac \alpha 2-\frac d2(\frac12 - \frac1r))k} [E]_\alpha^\frac12 \| G\|_{\ell_{\tau'}^{2}(\mathcal E_j; L_x^{r'})}
\Ee
holds for $2\leq r < r_\ast$ and uniformly in $\tau$.  Assuming this for the moment, we prove \eqref{1st}. 
Observe that
\begin{align*}
	|\mathcal{B}_k (F, G)| & \le \sum_{\tau}\Big| \Big\langle (U_\tau^\gamma P_j)^* F(\tau), \sum_{\tau'} \chi_{k}(\tau -\tau')(U_{\tau'}^\gamma P_j)^* G(\tau')\Big\rangle_x\Big|.
	\end{align*}
	Thus, the  Cauchy--Schwarz inequality and \eqref{claim_endpoint}  give
	\begin{align*}
	|\mathcal{B}_k (F, G)|
	&\leq \sum_{\tau} \| F(\tau)\|_{L_x^2} \Big\| \sum_{\tau'}\chi_{k}(\tau - \tau')(U_{\tau'}^\gamma P_j)^* G(\tau') \Big\|_{L_x^2}
	\\
	&\le C  2^{d(\frac12 - \frac1r)j} 2^{( \frac \alpha 2-\frac d2(\frac12 - \frac1r))k} [E]_\alpha^\frac12  \|F\|_{\ell_\tau^{1}(\mathcal E_j; L_x^2)} 
	 \| G\|_{\ell_{\tau'}^{2}(\mathcal E_j; L_x^{r'})}. 
\end{align*}
Note that we are assuming $\supp F(x,\cdot)$ is contained in an interval $J$ of length  $\sim 2^k2^{-\gamma j }$. 
Since  $\mathcal E_j$ is a $2^{-\gamma j}$-separated  subset of $E$ with bounded $\alpha$-Assouad characteristic, as before,  we have $\#\{\tau\in \mathcal E_j:  \tau\in J\}\le C 2^{\alpha k} [E]_\alpha$. Thus,  H\"older's inequality gives
$\|F\|_{\ell_\tau^{1}(\mathcal E_j; L_x^2)} \le C2^{ \frac\alpha 2 k }[E]_\alpha^\frac12 \|F\|_{\ell_\tau^{2}(\mathcal E_j; L_x^2)}$, and hence \eqref{1st} follows. 

It remains to verify the claim \eqref{claim_endpoint}.  Recalling  \eqref{utut} and \eqref{chi-k} and repeating the similar argument as before ({\it cf.}  \eqref{hoho}, 
we note that 
\eqref{claim_endpoint} follows if we show 
\begin{align*}
	\Big|  \sum_{\tilde \tau\in {\mathcal E}_j}  \Big\langle  \sum_{ \tau'\in{\mathcal E}_j }  \tilde \chi_k(\tilde \tau, \tau') T_{\tilde \tau, \tau'} G(\tau'), G(\tilde\tau)\Big\rangle\Big|
	\lesssim  2^{2d(\frac12 - \frac1r)j } 2^{(\alpha- d(\frac12 - \frac1r))k}  [E]_\alpha \| G\|_{\ell_{\tau'}^{2}({\mathcal E}_j; L_x^{r'})}^2.
\end{align*}
for $2\le r\le r_\ast$, where 
\[    \tilde \chi_k(t, s)=  \chi_k (\tau - t) \chi_k (\tau - s).\]  

By applying  H\"older's,  Minkowski's and Cauchy--Schwarz's inequalities successively and following the same argument as in the proof of Proposition~\ref{lem_discrete_Str}, one sees  that  the matter is reduced to showing 
\[  \Big\| \sum_{ \tau' \in  {\mathcal E}_j }   \tilde \chi_k(\tilde \tau, \tau')  \|T_{\tilde\tau, \tau'} G(\tau')\|_{L_x^{r}} \Big\|_{\ell_{\tilde \tau}^{2}({\mathcal E}_j )} \lesssim [E]_\alpha 2^{2d(\frac12 - \frac1r)j } 2^{(\alpha- d(\frac12 - \frac1r))k} \| G\|_{\ell_{\tau'}^{2}({\mathcal E}_j ; L_x^{r'})}. \]
By \eqref{250309_1101}, this in turn follows from 
\Be 	 
\label{rr-22}
\Big\| \! \sum_{ \tau' \in \mathcal E_j }   \tilde \chi_k(\tilde \tau, \tau') \mathcal K_r(\tilde\tau, \tau') \|G(\tau')\|_{L_x^{r'}} \Big\|_{\ell_{\tilde \tau}^{2}({\mathcal E}_j )}\! \lesssim \![E]_\alpha 2^{(\alpha- d(\frac12 - \frac1r))k} \| G\|_{\ell_{\tau'}^{2}({\mathcal E}_j ; L_x^{r'})}. 
\Ee

Note that $ \tilde \chi_k(\tilde \tau, \tau') =0$ if $|\tilde \tau- \tau'|\ge 2^{k-\gamma j+2}$.  Using the same decomposition as before,  by \eqref{gr} and \eqref{ak}, we obtain 
\[ 
 \|  \tilde \chi_k(\tilde \tau, \cdot)  \mathcal K_r(\tilde \tau, \cdot)\|_{\ell^1_{\tau'}(\mathcal E_j)} \le \! \! \sum_{0\le l\le k+2} \! \| \chi_l (\tilde \tau, \cdot)  \mathcal K_r(\tilde \tau, \cdot) \|_{\ell^1_{\tau'}(\mathcal E_j)} \lesssim \!  [E]_\alpha \! \! \sum_{0\le l\le k+2}  \!  2^{(\alpha- d(\frac12 - \frac1r))l }.
 \]
Since $r<r_\ast$,  it follows that 
\[\sup_{\tilde \tau \in  \mathcal E_j } \| \tilde \chi_k(\tilde \tau, \cdot) \mathcal K_r(\tilde \tau, \cdot)   \|_{\ell_{\tau'}^{1}({\mathcal E}_j )} \lesssim [E]_\alpha 2^{(\alpha- d(\frac12 - \frac1r))k}.   \] 
By symmetry the same estimate also holds for  $\sup_{ \tau'\in  \mathcal E_j }   \| \tilde \chi_k(\cdot, \tau') \mathcal K_r(\cdot, \tau')   \|_{\ell_{\tilde \tau}^{1}({\mathcal E}_j )}$.  Therefore, from  {\bf (A)} in  Lemma \ref{young} with $p=q=2$,  we obtain the desired inequality. 
\end{proof}

\begin{rem}\label{interval}  Let $1\le \tilde q'\le q\le \infty$.    Suppose the estimate  
\Be
\label{Bk}
| \mathcal{B}_k (F,G) |\le B \|F\|_{\ell^{q'}(\mathcal E_j; X)} \|G\|_{\ell^{\tilde q'}(\mathcal E_j; Y)}
\Ee 
holds whenever   $\supp F(x, \cdot)$ and $\supp  G(x, \cdot)$  are contained in an interval of length $2^{k-\gamma j +2}$ for all $x$.  
Here $X$ and $Y$ denote appropriate $L_x^p$ spaces.   Then, the same estimate \eqref{Bk} remains valid without the  assumption on the support of $F, G$. 

Indeed, let  $\mathcal{I}=\{I\}$ be a family  of disjoint intervals $I$ with $|I|=2^k 2^{-\gamma j }$ partitioning $\mathbb R$. Set 
\[ F = \sum_{I\in \mathcal I} F_I, \quad G = \sum_{I\in \mathcal I}G_{I},\] where $F_I(x,t)=F(x,t)\chi_I(t)$ and $G_{I}(x,t)= G(x,t)\chi_I(t)$.
Then,  by the assumption  it follows that 
\begin{align*}
	| \mathcal{B}_k (F,G) | &\le \sum_{I, I'\in \mathcal I: \dist(I,I')\le 2^k2^{-\gamma j }}  |\mathcal{B}_k (F_I,G_{I'}) |.
\\ 
 &\le \sum_{I, I'\in \mathcal I: \dist(I,I')\le 2^k2^{-\gamma j }} B \|F_I\|_{\ell^{q'}(\mathcal E_j; X)} \|G_{I'}\|_{\ell^{\tilde q'}(\mathcal E_j; Y)},
\end{align*}
Since $I$ are disjoint intervals of length $2^{k-\gamma j}$,   H\"older's inequality and disjointness of the intervals yield
\begin{align*}
	| \mathcal{B}_k (F,G) | 
	   &\lesssim  B  \Big( \sum_{I \in \mathcal I} \|F_I\|_{\ell^{q'}(\mathcal E_j; X)}^{q'}  \Big)^{1/{q'}}   \Big( \sum_{I \in \mathcal I}  \|G_I\|_{\ell^{\tilde q'}(\mathcal E_j; Y)}^{q}\Big)^{1/q} 
	   \\
	    &\le B   \|F\|_{\ell^{q'}(\mathcal E_j; X)}    \|G\|_{\ell^{\tilde q'}(\mathcal E_j; Y)} .
\end{align*}
For the second inequality we use the fact that $\tilde q'\le q$. 
\end{rem}

We now use the estimate in Lemma \ref{lem_interpol1}  to get the endpoint estimate via bilinear interpolation.

\subsubsection*{Bilinear interpolation} 
For a sequence $ \mathbf a=\{a_k \}_{k\ge 0}$, for $1\le p\le \infty$,  define 
\[ 
\| \mathbf a\|_{\tilde\ell_\mu^p} = 
\begin{cases}
\Big( \sum_{k\ge 0} ( 2^{\mu k } |a_k|)^p \Big)^\frac1p,  \  & p\neq\infty,
\\
 \ \  \ \sup_{k\ge 0} 2^{\mu k } |a_k|,  \ & p=\infty. 
\end{cases}     \]

We consider a vector-valued bilinear operator that is defined by 
\[
	\mathbf B(F,G) = \Big\{\frac{2^{(\lambda(a,b)-\alpha)k}}{[E]_\alpha}  \mathcal{B}_k(2^{\frac d2 j} F(2^j \cdot, \cdot),2^{\frac d2 j} G(2^j \cdot, \cdot)\Big\}_{k\ge 0}.
\] 
Though not strictly necessary,  the scaled inputs $2^{\frac{d}{2} j} F(2^j \cdot, \cdot)$ and $2^{\frac{d}{2} j} G(2^j \cdot, \cdot)$ are introduced for simplicity, to eliminate the dependence on $j$ in the operator bounds.    By Lemma~\ref{lem_interpol1} and scaling, we have  
\Be\label{vv_operator}
\|\mathbf B(F,G)\|_{\tilde\ell_\infty^{\lambda(a,b)-\alpha}(\mathbb{N}_0)}\le C\|F \|_{\ell_\tau^2(\mathcal E_j; L_x^{a'})}\|G \|_{\ell_\tau^2(\mathcal E_j; L_x^{b'})} 
\Ee
for  $(1/{a}, 1/{b})$ contained in the interior of $\mathcal Q$.  That is to say, 
\[  
\mathbf B: \ell_\tau^2(\mathcal E_j; L_x^{a'}) \times \ell_\tau^2(\mathcal E_j; L_x^{b'}) \to \tilde\ell^\infty_{\lambda(a,b)-\alpha}
\] 
is bounded  for all $(1/a, 1/b)$ contained in the interior of $\mathcal Q$ (the closed quadrangle with vertices in \eqref{aa}). 

By scaling and \eqref{sumk}, the estimate \eqref{goal}  follows if we show  
\Be
\label{bT}
	\mathbf B :  \ell_\tau^2(\mathcal E_j; L_x^{r_\ast'}) \times \ell_\tau^2(\mathcal E_j; L_x^{r_\ast'})  \to \tilde\ell^1_{0}.
\Ee
To this end, we apply a bilinear real interpolation result (\cite{BerLof}) (see  \cite{Janson} for a multilinear generalization). 
\begin{lem}\label{lem_bi_interpol}\cite[Section 3.13.5(b)]{BerLof}
	Let $A_0, A_1, B_0, B_1, C_0, C_1$ be Banach spaces, and the bilinear operator $T$ be such that 
	\begin{align*}
		T : A_{i} \times B_j \to C_{i+j}
	\end{align*}
	is  bounded for $(i,j)=(0,0),$ $(0,1),$ and $(1,0)$. 
	Then, for $0<\alpha_0, \alpha_1, \alpha = \alpha_0+\alpha_1<1$, $1\leq p,q,r\leq\infty$, and $1\leq 1/p+ 1/q$, 
	\[
		T : (A_0, A_1)_{\alpha_0, pr} \times (B_0, B_1)_{\alpha_1, qr} \to (C_0, C_1)_{\alpha, r}
	\]
	is bounded. 
Here, $(A_0, A_1)_{\alpha, p}$	denotes the real interpolation space (in the sense of the $
K$-method) between $A_0$ and  $A_1$.
\end{lem}

We choose  $r_0, r_1,$ such that $(1/r_0, 1/r_0),   (1/r_0, 1/r_1), (1/r_1, 1/r_0)\in \operatorname{int} \mathcal Q$ (see Figure~\ref{fig3}), and 
\[   
\lambda(r_0, r_0)-\alpha >0> \lambda(r_0, r_1) -\alpha.
\]  
Note that $\lambda(r_0, r_1)=\lambda(r_1, r_0)$ by symmetry.  Also, note that 
$   \lambda(r_\ast, r_\ast)=\alpha$, thus   $(1/r_\ast, 1/r_\ast)\in \operatorname{int} \mathcal Q$ (see \eqref{r*}).

\begin{figure}[t]
\centering

\begin{tikzpicture}[scale=4]
  % Axes
  \draw[->] (0,0) -- (1.05,0) node[right] {\( \frac{1}{\tilde r} \)};
  \draw[->] (0,0) -- (0,1.05) node[above] {\( \frac{1}{r} \)};
  
  % Unit square boundary
  \draw[thick] (0,0) rectangle (1,1);

  % Dashed reference lines at 1/2
 
  \node at (1.0,-0.1) {\( \frac{1}{2} \)};
  \node at (-0.1,1.0) {\( \frac{1}{2} \)};
  \node at (-0.1, 0.575) {\( \frac{1}{r_0}\)};
  \node at (0.575, -0.1) {\(\frac{1}{r_0}\)};
  \node at (0.8, -0.1) {\(\frac{1}{r_1}\)};
  \node at (-0.1, 0.8) {\(\frac{1}{r_1}\)};

  % Coordinates
  \coordinate (A) at (1.0, 1.0);      % (1/2, 1/2)
  \coordinate (B) at (1.0, 0.625);    % (1/2, 3/8)
  \coordinate (C) at (0, 0);          % (0, 0)
  \coordinate (D) at (0.625, 1.0);   % (5/16, 1/2)
  \coordinate (P1) at (0.625, 0.625);       % Point on line
  \coordinate (R0) at (0.575, 0.575);
  \coordinate (R1) at (0.575, 0.8);
  \coordinate (R2) at (0.8, 0.575);
  \coordinate (P2) at (25/104, 5/13);      % upper intersection
  \coordinate (P3) at (5/13, 25/104);      % lower intersection

  % Fill and draw the quadrilateral
  \fill[blue!20] (R0) -- (R1) -- (R2) -- cycle;
  \draw[blue, thick] (A) -- (B) -- (C) -- (D) -- cycle;
  
  \draw[blue, thick] (R0) -- (R1) -- (R2) -- cycle;

  % Draw dots
  \filldraw (B) circle (0.01);
  \filldraw (D) circle (0.01);
  \filldraw (R0) circle (0.01);
  \filldraw (R1) circle (0.01);
  \filldraw (R2) circle (0.01);
  \filldraw (P1) circle (0.01);

  % Label the points with coordinates
  \node[right] at (B) {$\frac{1}{r_*}$};
  \node[below left] at (C) {O};
  \node[above] at (D) {$\frac{1}{r_*}$};
   
   \draw[black, densely dotted] (C) -- +(1.0,1.0);
   \draw[black, densely dotted] (P1) -- +(0.0,0.375);
   \draw[black, densely dotted] (P1) -- +(0.375,0.0);
   \draw[black, densely dotted] (R0) -- +(0.0,-0.575);
     \draw[black, densely dotted] (R0) -- +(-0.575, 0.0);
     \draw[black, densely dotted] (R2) -- +(0.0, -0.575);
        \draw[black, densely dotted] (R1) -- +(-0.575, 0.0);
\end{tikzpicture}
\caption{Bilinear interpolation for $(1/r_*, 1/r_*)$}
\label{fig3}
\end{figure}
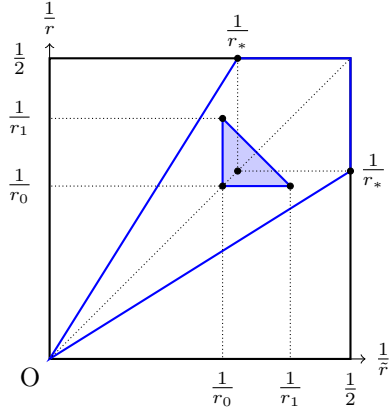

We apply Lemma \ref{lem_bi_interpol},  taking 
\begin{align*}
	A_0, B_0&=  \ell_\tau^2(\mathcal E_j; L_x^{r_0'}),  \hspace{-50pt} &A_1, B_1 &=  \ell_\tau^2(\mathcal E_j; L_x^{r_1'}), 
	\\
	 C_0&= \tilde\ell^\infty_{\lambda(r_0,r_0)-\alpha},  \hspace{-50pt} &C_1 &= \tilde\ell^\infty_{\lambda(r_0,r_1)-\alpha}, 
\end{align*}
and $p=q=2$, $r=1$.
Then, by combining the estimate \eqref{vv_operator}  and  Lemma~\ref{lem_bi_interpol}, 
 it follows that the operator 
 \Be
 \label{hohoho}
\begin{aligned}
\mathbf B:  ( \ell_\tau^2(\mathcal E_j; L_x^{r_0'}),  & \ell_\tau^2(\mathcal E_j; L_x^{r_1'}))_{\alpha_0, 2} \times ( \ell_\tau^2(\mathcal E_j; L_x^{r_0'}), \ell_\tau^2(\mathcal E_j; L_x^{r_1'}))_{\alpha_1, 2} 
\\
&\to (\tilde\ell^\infty_{\lambda(r_0,r_0)-\alpha}, \tilde \ell^\infty_{\lambda(r_0, r_1)-\alpha})_{\alpha, 1}
\end{aligned}
\Ee
is bounded for $0<\alpha_0, \alpha_1,$ $\alpha = \alpha_0+\alpha_1<1$.  Moreover, it is well known that  the real interpolation yields 
\[ 	(\ell_\tau^2(\mathcal E_j;  L_x^{p_0}), \ell_\tau^2(\mathcal E_j; L_x^{p_1}))_{\alpha, 2} = \ell_\tau^2(\mathcal E_j; L_x^{p,2}),\quad
	(\tilde\ell^\infty_{\lambda_0}, \tilde\ell^\infty_{\lambda_1})_{\alpha, 1} = \tilde\ell^1_\lambda,\] 
where $\frac1p = \frac{1-\alpha}{p_0} + \frac\alpha {p_1}$ and $\lambda = (1-\alpha)\lambda_0 + \alpha\lambda_1$.
Combining this and \eqref{hohoho} yields 
\Be
\label{bTT}
\mathbf B:   \ell_\tau^2(\mathcal E_j; L_x^{a', 2}) \times  \ell_\tau^2(\mathcal E_j; L_x^{b', 2})  \to   \ell^1_{\lambda(a,b)-\alpha}
\Ee
for $(1/a, 1/b)$ contained in the interior of the triangle with vertices 
$(1/r_0, 1/r_0),$ $(1/r_0, 1/r_1),$ $(1/r_1, 1/r_0)$ (see Figure~\ref{fig3}).  
In particular, taking $(a, b)=(r_\ast, r_\ast)$, we get \eqref{bT}.  This completes the proof of 
\eqref{250309_0149} for  $(q,r)=(2, r_\ast)$.

\begin{rem}\label{wave-homo} Let $d\ge 2$. The same argument also works for the wave operator, i.e., $\gamma=1$. 
 When $\gamma=1$, the stationary phase method gives
\Be
\label{dis-wave}
	\| T_{\tau, \tau'} g \|_{L_x^\infty} \le C 2^{dj}  (1+ 2^{\gamma j }|\tau - \tau'|)^{-\frac {d-1}2} \|g\|_{L_x^1}
	\Ee
 instead of \eqref{250309_1100}. Thus, a routine adaptation of the previous argument  proves 
 the estimate  \eqref{f-stri}  for $(q,r)\in \mathbb [2,\infty]\times [2,\infty]\setminus\{(2,\infty)\}$, provided that  \eqref{con-reg} holds with $\gamma=1$ and the inequality 
$
 	\frac {d-1}2  (\frac 12 - \frac1r) \ge  \frac \alpha q
$
 holds.  
\end{rem}

\subsection{Proof of Theorems \ref{cor_main} and \ref{cor_m}} 
\label{sec:2-1}
As in Section \ref{sec:2-3},  the proof reduces to establishing the following.

\begin{prop}\label{lem_discrete_Str1} Let $\gamma\in \mathbb R_{>0} \setminus \{1\}$,  $2\le q,r \leq \infty$ and $(q,r)\neq (2,\infty)$.   
Suppose that   $\mu$ is a  measure  satisfying \eqref{alpha-dim}. Then,  for $j\in \mathbb Z$,  we have 
 \Be \label{f-stri1p}   \| U^\gamma_t P_j  f\|_{L_t^q( \mathrm{d}\mu ; L_x^r)} \le C2^{(\frac d2 - \frac dr-\frac{\gamma \alpha}q)j} \langle \mu \rangle_\alpha^{1/q} \|f\|_{2}\Ee
  provided that \eqref{con-ad} is satisfied.   
 \end{prop}

By \eqref{alpha-dim}, the total mass $\|\mu\|$ of $\mu$ is bounded by $(\diam\supp \mu)^\alpha \langle \mu \rangle_\alpha$. 
Once Proposition \ref{lem_discrete_Str1} is established,    Theorem~\ref{cor_main} is a straightforward consequence of the following lemma.

\begin{lem}\label{lem_freq_loc} Let $\gamma\in \mathbb R_{>0} \setminus \{1\}$, $(q,r)\in \mathbb [2,\infty]\times [2,\infty)$ and $s\in \mathbb R$.
	Let $\mu$ be a positive  measure satisfying \eqref{alpha-dim} with its support  bounded.  Suppose that the estimate    
	\Be \label{assum} \|U_t^\gamma P_j f \|_{L_t^q(\mathrm{d}\mu; L_x^r)} \leq B 2^{js} \| f\|_2\Ee
	holds for $j\ge 1$. 	Then, we have
	\[
		\|U_t^\gamma f \|_{L_t^q(\mathrm{d}\mu; L_x^r)} \leq  C ( \|\mu\|^{1/q} + B )\|f\|_{H^s(\mathbb{R}^d)}.
	\]
\end{lem}

\begin{proof} Since ${\rm Id}=P_{\le 0}+ \sum_{j>0} P_j$, we have 
\[\|U_t^\gamma f \|_{L_t^q(\mathrm{d}\mu; L_x^r)} \le  \|U_t^\gamma P_{\le 0} f \|_{L_t^q(\mathrm{d}\mu; L_x^r)}+ \|\sum_{j>0} U_t^\gamma  P_jf \|_{L_t^q(\mathrm{d}\mu; L_x^r)} .\]
Since $r\ge 2$, using  Berstein's inequality and Plancherel's theorem, we note that $\|U_t^\gamma P_{\le 0} f \|_{L_x^r}\le C\|P_{\le 0} f \|_{L_x^2(\mathbb{R}^d)}$. 
Thus,  $\|U_t^\gamma P_{\le 0} f \|_{L_t^q(\mathrm{d}\mu; L_x^r)}\le C \|\mu\|^{1/q} \|f\|_{H^s(\mathbb{R}^d)}$, so we need only  to handle the remaining part. 

Using  the Littlewood--Paley inequality, we have
\[ \|\sum_{j>0} U_t^\gamma  P_jf \|_{L_x^r} \le C\Big\| \Big(\sum_{j>0} |U_t^\gamma P_j f|^2 \Big)^{1/2}\Big\|_{L^r_x(\mathbb{R}^d)}.\]
Since $2\le r, q$,  Minkowski's inequality yields 
	\begin{align*}
		 \|\sum_{j>0} U_t^\gamma  P_jf \|_{L_t^q(\mathrm{d}\mu; L_x^r)} 
		&\leq C \Big(\sum_{j>0} \| U_t^\gamma P_j f\|_{L_t^q(\mathrm{d}\mu; L_x^r)}^2 \Big)^{1/2}.
	\end{align*}
	Applying the assumption \eqref{assum} and using $\tilde P_j$,  we have 
	\begin{align*}
		 \|\sum_{j>0} U_t^\gamma  P_jf \|_{L_t^q(\mathrm{d}\mu; L_x^r)} 
		\lesssim  B \Big(\sum_{j>0} 2^{2js} \| \tilde P_j f\|_{L^2_x}^2 \Big)^{1/2}\le B \|f\|_{H^s(\mathbb{R}^d)}.
	\end{align*}
	This completes the proof. 
\end{proof}

Using Proposition \ref{lem_discrete_Str1} and the Littlewood--Paley inequality as above, one can also prove Theorem \ref{cor_m} in a similar manner as Theorem \ref{homo-homo}. We omit the details.

We now turn to the proof of Proposition \ref{lem_discrete_Str1}. 
As before, the estimates \eqref{f-stri1p}  in Proposition \ref{lem_discrete_Str1}  for $j\neq 0$ can be deduced from that for $j=0$ by controlling 
 $\langle\cdot \rangle$ of the associated measures arising after scaling. 
 Nevertheless, in what follows we establish \eqref{f-stri1p} while keeping dependence on $j$ explicit.

\subsubsection{Proof of Proposition \ref{lem_discrete_Str1}} 

In order to show \eqref{f-stri1p}, we follow the same line of argument as in the previous sections, so we shall be brief.  
As before, we may assume $r > 2$.  
Recalling the condition \eqref{con-ad}, we again distinguish the admissible pairs $(q,r)$ into two cases: the non-endpoint case $(q,r) \neq (2, r_\ast)$ and the endpoint case $(q,r) = (2, r_\ast)$.  
We treat the non-endpoint case ({\it cf.} ~Section~\ref{non-end}) and the endpoint case ({\it cf.} ~Section~\ref{sec_endpoint}) separately.
 
We consider a bilinear operator 
\[
	\tilde{\mathcal B}(F,G) = \iint \langle U_s^\gamma P_j F (s), U_t^\gamma P_j G(t)\rangle_x \,\mathrm{d}\mu (s) \mathrm{d}\mu(t).
\]

Note that the estimate \eqref{f-stri1p} is equivalent to 
\Be 
\label{bi1}
|\tilde{\mathcal B}(F,F)|  \le C 2^{2(\frac d2 - \frac dr-\frac{\gamma \alpha}q)j} \langle \mu \rangle_\alpha^{2/q}  \| F \|_{L_t^{q'}(\mathrm{d}\mu; L_x^{r'})}^2. 
\Ee
({\it cf.} \eqref{hoho} and  \eqref{goal}). Recalling \eqref{utut},  we note that \eqref{bi1} follows 
 if we show 
\[
\Big\|  \int T_{t, s}  F (s)\,  \rd \mu(s) \Big\|_{L^{q}_t(\rd\mu; L^{r}_x)}
	\leq C 2^{2(\frac d2 - \frac dr-\frac{\gamma \alpha}q)j}  \langle \mu \rangle_\alpha^{2/q} \|F\|_{L^{q'}_t( \rd\mu; L^{r'}_x)}. 
\]
Now, by  \eqref{250309_1101},  the above estimate  reduces to
\Be
	\Big\|  \int  \mathcal K_r(t, s) \| F(s)\|_{L_x^{r'}} \,  \rd \mu(s) \Big\|_{L^{q}_t(\rd\mu)}
	\le C  2^{ -\frac{2\gamma \alpha}qj}  \langle \mu \rangle_\alpha^{2/q} \| F\|_{L^{q'}_t(\rd\mu; L_x^{r'})}.
\Ee

This is a consequence of the next lemma, which can be shown by a simple modification of the proof of  Lemma \ref{kernel}.

\begin{lem} \label{kernel1} Let $r>2$, and  $\mathcal K_r$  be given by \eqref{gr}. If  $d(\frac1 2 - \frac1 r)> \frac \alpha \sigma $, then we have
\Be 
\label{strong1}
\sup_{s}  \|\mathcal K_r(\cdot, s)\|_{L^{\sigma}_t(\rd\mu)}, \,\,  \sup_{t}    \|\mathcal K_r(t,\cdot)\|_{L^{\sigma}_t(\rd\mu)} \le C \langle \mu\rangle_\alpha^{1/\sigma} 2^{-\frac{\gamma\alpha}\sigma j}
\Ee
for a constant $C$, independent of $j$.   Moreover,  if $d(\frac1 2 - \frac1 r)= \frac \alpha \sigma $, then 
\Be \label{weak1}
\sup_{s}  \|\mathcal K_r(\cdot, s)\|_{L^{\sigma,\infty}_t(\rd\mu)}, \,\,  \sup_{t}    \|\mathcal K_r(t,\cdot)\|_{L^{\sigma,\infty}_t(\rd\mu)} \le C \langle \mu\rangle_\alpha^{1/\sigma} 2^{-\frac{\gamma\alpha}\sigma j}
\Ee for a constant $C$, independent of $j$.  
\end{lem}

\begin{proof}  Recalling  \eqref{gr} and \eqref{chi-k}
\[\mathcal K_r(t, s)^\sigma\le  \sum_{k=0}^\infty \mathcal K_{r,k}(t, s)^\sigma   \le  C \sum_{k\geq0} 2^{-k d(\frac1 2 - \frac1 r)\sigma} \chi_k(t-s).\]
Since $\supp \chi_k(t, \cdot)$ is contained in an interval of length $2^{k-\gamma j+1}$, using \eqref{alpha-dim} we get
\[\int \mathcal K_r(t, s)^\sigma\, \rd \mu(s) \le  C\langle \mu\rangle_\alpha \sum_{k\geq0} 2^{-k d(\frac1 2 - \frac1 r)\sigma}  2^{(k-\gamma j)\alpha}.\]
This gives \eqref{strong1} if  $d(\frac1 2 - \frac1 r)> \frac \alpha \sigma $. On the other hand, as we have seen in the proof of Lemma \ref{kernel},  note that  $\{  t :  \mathcal K_r(t, s)>\lambda \}$ is contained in an interval of length $C  2^{-\gamma j}  \lambda^{- 1/d(\frac1 2 - \frac1 r)}$. Thus, \eqref{alpha-dim} yields \eqref{weak1}. 
\end{proof}

It is easy to see that  
a modification of Lemma \ref{young}  continues to hold even after 
 $\sum_\tau$ and $ \ell^{\sigma}_\tau (\mathcal E)$  are replaced by $\int  \,\mathrm d\mu(s)$ and $L^{\sigma} (\rd \mu)$. Thus, combining this and Lemma \eqref{kernel1} and  following the argument in Section \ref{sec:2-3},  one can show the non-endpoint case of the estimate \eqref{f-stri1p}, that is to say, for $(q,r)$ satisfying \eqref{con-ad}  with $(q,r)\neq (2, r_\ast)$. Note that $\langle \mu \rangle_\alpha^{2/q}$ of \eqref{f-stri1p} follows by applying Lemma~\ref{kernel1} with $1+1/q = 1/\sigma + 1/q'$, which yields $1/\sigma = 2/q$.

We now consider the endpoint case.  Decompose 
\[ \tilde {\mathcal{B}} (F,G)= 
\sum_{k=0}^\infty   \tilde {\mathcal{B}}_k (F,G), \]
where
\begin{align*}
	\tilde {\mathcal{B}}_k (F, G) =  \iint \chi_k (t-s) \langle T_{t, s}  F(t), G(s) \rangle_x \, \rd\mu(t) \rd\mu(s),  \quad k\ge 0.  
\end{align*}
Similarly as in Section \ref{sec_endpoint}, 
we obtain the following ({\it cf.} Lemma \ref{lem_interpol1}).

\begin{lem}\label{lem_interpol11}
Let  $1\le a, b\le \infty$. 
If  $(1/{a}, 1/{b})$ is contained in the interior of the quadrangle $\mathcal Q$ $($with vertices $(0,0)$,  $A=(1/2, 1/r_\ast)$,   $(1/2, 1/2)$, and  $A'=(1/r_\ast, 1/2),$ see \eqref{r*})
  we have
	\[ 		|\tilde {\mathcal{B}}_k (F, G)| \leq C 2^{(2\lambda(a,b)- \gamma \alpha) j } 2^{-(\lambda(a, b)-\alpha)k} \langle \mu \rangle_\alpha \|F \|_{L^2(\rd\mu; L_x^{a'})}\|G \|_{L^2(\rd\mu; L_x^{b'})}.\]
\end{lem}

\begin{proof} 
Examining the proof of Lemma~\ref{lem_interpol1} (see, e.g., \eqref{00-22}) shows that  
the estimate \eqref{k-loc} for $(a,b) = (2,2)$ and $(\infty,\infty)$ can be obtained by the same argument, once we have  
\[
    \sup_{s} \|\mathcal K_{r,k}(\cdot, s)\|_{L^{1}(\mathrm{d}\mu)},  
    \,\,  
    \sup_{t} \|\mathcal K_{r,k}(t, \cdot)\|_{L^{1}(\mathrm{d}\mu)}   
    \le C\, 2^{-\gamma \alpha j} \, 2^{(\alpha - \lambda(r,r))k} \langle \mu \rangle_\alpha,
\]
which follows easily from \eqref{K-chi} and \eqref{alpha-dim}.   Similarly, for $2 \le r < r_\ast$, we obtain \eqref{k-loc} for $(a,b) = (2,r)$ and $(a,b) = (r,2)$ using the estimate  
\[
    \sup_{t} \| \tilde{\chi}_k(t, \cdot) \mathcal K_r(t, \cdot) \|_{L^{1}(\mathrm{d}\mu)},  
    \,\,
    \sup_{s} \| \tilde{\chi}_k(\cdot, s) \mathcal K_r(\cdot, s) \|_{L^{1}(\mathrm{d}\mu)} 
    \lesssim 2^{-\gamma \alpha j} \, 2^{\left(\alpha - d\left(\frac12 - \frac1r\right)\right)k} \langle \mu \rangle_\alpha.
\]
(See the argument below \eqref{rr-22}.)  
This can be shown in the same manner as in the proof of Lemma~\ref{lem_interpol1}. Interpolating between these estimates yields all the desired bounds.  
\end{proof} 

Finally, once we have Lemma \ref{lem_interpol11}, the same bilinear interpolation argument in Section \ref{sec_endpoint} proves  the estimate \eqref{f-stri1p} for $(q,r)=(2, r_\ast)$.   We omit the details.

\begin{rem}\label{wave-homo1} 
Using \eqref{dis-wave} and  the same argument, one can obtain 
 the estimates   \eqref{f-stri10} and    \eqref{f-stri1} for  the wave operator, i.e., $\gamma=1$, whenever $(q,r)\in \mathbb [2,\infty]\times [2,\infty)$ 
  \Be\label{con-ad-w}
 	\frac {d-1}2 \Big(\frac 12 - \frac1r\Big) \ge  \frac \alpha q. 
 \Ee
\end{rem}

\subsection{Necessity of the conditions in Theorems~\ref{thm_main} and \ref{cor_main}}\label{subsec_nec}
Unlike the usual Strich\-artz estimates, the scaling argument does not give necessary conditions on the exponents $r$, $q$, and $s$. 
Instead, one needs to exploit the notion of bounded $\alpha$-Assouad characteristic and the $\alpha$-dimensional measure by considering specific sets $E$.

\subsubsection{Necessity of \eqref{con-ad} and  \eqref{con-reg} for the estimate \eqref{f-stri}} 
We first consider \eqref{con-reg}, which is easier to show.

\begin{prop}\label{thm-nec}  Let $\gamma\in (0, \infty)$. Suppose that   the estimate  
\eqref{f-stri}  holds  for  any nonempty subset $E$. Then, we have \eqref{con-reg}. 
 \end{prop}

\begin{proof}  Note that $E(2^{-\gamma j})\cap [2^{-1-\gamma j}, 1-2^{-1-\gamma j}]\neq \emptyset$ as far as $E$ is not empty. 
Let $t_\circ\in  [2^{-1-\gamma j}, 1-2^{-1-\gamma j}]$. 
For $j \gg1$, we take $f$ given by 
\[ \widehat{f}(\xi) = \psi(\xi/2^j) e^{-it_\circ|\xi|^\gamma}.\] Changing variables, we see that 
\begin{align*}
		2^{d j} \chi_{B(0, c2^{-j})}(x) \chi_{(-c2^{-\gamma j}, c2^{-\gamma j})}(t-t_\circ) \leq C | e^{ i t(-\Delta)^{\gamma/2} } f (x) |
	\end{align*}
	with a positive constant $c$ small enough.  
	Then, the estimate \eqref{f-stri} implies 
	\[
		2^{d j} 2^{-j(d/r+ \gamma/q)} \le C 2^{j(s+{d/2)}}. 
	\]
	Letting $j \to \infty$, we obtain   the condition \eqref{con-reg}.
	\end{proof}

	In order to show that  \eqref{f-stri}  fails unless \eqref{con-ad} holds, we make use of a specific set with bounded $\alpha$-Assouad characteristic and a measure supported on it. For the purpose, we use the following lemma. 
	
	 \begin{lem}\label{ex}
	 Let \( \alpha \in (0,1) \). Then, there exists a Borel set \( E \subset [0,1] \) with bounded \( \alpha \)-Assouad characteristic such that$:$
\begin{itemize}[leftmargin=2.em, itemsep=0.5em, topsep=-0.25em]
    \item[$\rm(A)$] For integers \( l > k \ge 0 \),  there are  a \( 2^{-l} \)-separated subset \( E_l \subset E \)  and  an interval $J\subset [0, 1]$ of length 
    $\sim 2^{k-l}$  such that 
    \[
    \# (E_l  \cap J) \sim 2^{\alpha k}. 
    \]
      \item[$\rm(B)$] There exists an Ahlfors--David \( \alpha \)-regular Borel measure \( \mu \) with  \( \supp \mu=E \), that is,
    \Be
    \label{mess}
    \mu(B(x,r)) \sim r^\alpha \quad \text{for all } x \in \mathrm{supp} \mu \text{ and } 0 < r \le 1.
    \Ee
\end{itemize}
\end{lem}

\begin{proof}
For $\rm (A)$, we construct a Cantor set following the standard construction in \cite[p. 60]{Mat} (see also  \cite[Lemma 4.1]{LRZZ}).
For the base step, say step 0, let $I_{0, 1}^\alpha = [0,1]$.

In step $1$, continuously, we take $I_{1, 1}^\alpha = [0, 2^{-1/\alpha}]$ and $I_{1, 2}^{\alpha} = [1-2^{-1/\alpha},1]$, where we exclude the intermediate interval of length $1 - 2^{1 - 1/\alpha}$.
In step $k$, we remove the middle intervals of length 
$$
	|I_{k-1, j}^\alpha| - 2^{-k/\alpha +1} = 2^{-(k-1)/\alpha} - 2^{-k\alpha +1}, \quad j\in\{1, \dots, 2^k\}
$$	
from intervals of step $(k-1)$.
In step $k$, therefore, one has $I_{k, 1}^\alpha, \dots, I_{k, 2^k}^\alpha$ of length $2^{-k/\alpha}$.
As a result, we obtain a set $E$ of right endpoints of these intervals given by
\[
	E = \Big\{ 2^{-n/\alpha} + \sum_{m=0}^{n-1} (1-2^{-1/\alpha})2^{-m/\alpha} r_m : r_m = 0, 1,\,\,n\in \mathbb N \Big \}.
\]

For given $l\geq0$, we choose $\nu = \lfloor \alpha l \rfloor$ such that $2^{-\nu/\alpha} \sim 2^{-l}$.
Define $E_l$ be a collection of right endpoints from $I_{\nu, 1}^\alpha, \dots, I_{\nu, 2^\nu}^\alpha$ so that $\#(E_l) = 2^\nu$.
Hence, from the fact that $ 2^{-\nu/\alpha} \sim 2^{-l}$  it follows that $\#(E_l) \sim 2^{\alpha l}$.
Let $l>k\ge0$ and $J\subset [0,1]$ be an interval of length $2^{k-l}$ centered at $x\in E$.
Then, for $m = -\lfloor \alpha(k- l) \rfloor \in \mathbb N$, we have $2^{-m/\alpha +1} \ge 2^{k-l} \ge 2^{-m/\alpha}$.
Moreover, there exist $I_{m, i_j}^\alpha$, $j=1,\dots, 4$, such that they are only intervals of  step $m$ intersecting $J$ and 
\[
	I_{m, i_1} \subset J,
\]
or there exist $I_{m+1, i_j}^\alpha$, $j=1,\dots, 4$, with the same property. Thus, 
using $\nu \sim \alpha l$, one has $\#(E_l \cap I_{m, i_j}) \sim  2^{\nu-m}$. This yelds from $2^m \sim 2^{-(k-l)\alpha}$ that
$$
	\#(E_l \cap I_{m, i_j}) = 2^{\alpha k}.
$$
There are only finite $j$'s, so we have $\#(E_l \cap J) \sim 2^{\alpha k}$.

For $\rm (B)$, we assign a probability measure $P_1$ on $I_{1,1}^\alpha \cup I_{1,2}^\alpha$ such that
\[
	\mu_1(I_{1, i}^\alpha) = \frac12,\quad i=1,2.
\]
Consecutively, we define $\mu_k$ a probability measure on step $k$ such that
\[
	\mu_k(I_{k, i}^\alpha) = 2^{-k},\quad i=1, \dots, 2^k.
\]
Then by Prokhorov's theorem \cite{Pro1956}, there exists a subsequence $(\mu_{k_n})_n$ and probability measure $\mu$ such that $\mu_{k_n}$ converges weakly to $\mu$.
Note from  the construction of $\mu$ that $\mu(I_{m, i}^\alpha) = 2^{-m}$ for $i=1, \dots, 2^m$.
Therefore, for an interval $I$ of length $r$ centered at $x\in E$, there is $j\in\mathbb N$ such that
\[
	2^{-j/\alpha} \leq r \leq 2^{-j/\alpha +2}.
\]
Thus we have
\[
	4^{-1} r^\alpha \le 2^{-j} \le \mu \left( (x-r, x+r) \right) \leq 2^{-j + 2} \leq 4 r^\alpha.  \qedhere
\]
\end{proof}

\begin{prop}\label{nec-del}
Let $\gamma\in (0, \infty) \setminus \{1\}$. Suppose that \eqref{con-reg} holds with equality.  Assume that   the estimate  
\eqref{f-stri}   holds  whenever $E\subset [0,1]$ has bounded $\alpha$-Assouad characteristic. Then, the inequality \eqref{con-ad} holds. 
 \end{prop}

\begin{proof}
 Assuming $\gamma j\gg m$, we  apply $\rm (A)$ in Lemma \ref{ex} with
\[ l=\lfloor \gamma j \rfloor, \quad k=m-L\] for a large positive integer $L$. Thus, we have  a set $E$ of bounded $\alpha$-characteristic and an interval $J$   as in Lemma \ref{ex}.  The interval $J$ is  of length $ C_12^{m-\gamma j}$ and  includes  $C_2 2^{\alpha m}$ elements of  a $C_32^{-\gamma j}$-separated subset of $E$. 
%We denote by $\tilde E$ the set of those elements. 
Then, it is clear that 
\Be \label{mess0}
 | E(2^{-\gamma j})\cap J| \ge C 2^{\alpha m-\gamma j}  .\Ee

With a proper choice of $t_\circ,$ we may assume $J\subset (t_\circ, t_\circ+ c2^{m-\gamma j})\subset [0,1]$.  
Let $\xi_0\in \mathbb R^d$ be such that $ |\xi_0|=1$. 
Consider $f$ that is  given by 
\[ \widehat{f}(\xi) = \psi( 2^{\frac m2}| 2^{-j}\xi-\xi_0|) e^{-it_\circ|\xi|^\gamma}.\]
After translation  scaling $\xi\to 2^j\xi$ and  $\xi\to \xi+\xi_0$, we have  
\[ |U^\gamma_t P_j f(x)|\sim   2^{dj}  \Big| \int_{\mathbb R^d} \e^{i( 2^j x\cdot\xi + 2^{\gamma j} (t-t_\circ)\phi(\xi+\xi_0))} \psi(|\xi+\xi_0|) \psi( 2^{\frac m2}|\xi|) \,\mathrm{d}\xi\Big|, \]
where $\phi(\xi)=|\xi|^\gamma$.  Setting  $R(\xi)=\phi(\xi+\xi_0)-\phi(\xi_0)-\nabla\phi(\xi_0)\cdot \xi$ and   
\[  \Phi(x,t,\xi)=\big( 2^j x-  2^{\gamma j}(t-t_\circ)\nabla\phi(\xi_0)\big)\cdot\xi + 2^{\gamma j} (t-t_\circ)R(\xi), \]
we have 
\[
 |U^\gamma_t f(x)|\sim   2^{dj}  \Big| \int_{\mathbb R^d} \e^{i \Phi(x,t,\xi)} \psi(|\xi+\xi_0|) \psi( 2^{\frac m2}|\xi|) \,\mathrm{d}\xi\Big|.
\]

Since $\xi=O(2^{-m/2})$ and $R(\xi)=O(2^{-m})$ in the integral, we consequently have  
\Be
\label{lower-} |U^\gamma_t f(x)|\gtrsim 2^{dj}2^{-\frac d2 m},
\Ee
if 
$ |t-t_\circ|\le c2^{m-\gamma j}$ and $|x- 2^{(\gamma-1) j}(t-t_\circ) \nabla\phi(\xi_0)\big)| \le c2^{m/2-j}$  for a sufficiently small $c>0$.  
Thus, $\| U^\gamma_t P_j f\|_{L_x^r}\gtrsim C2^{dj}2^{-\frac d2 m} 2^{\frac d{2r} m-\frac dr j} $ for $t\in J$. Combining this with \eqref{mess0} yields  
\[  2^{dj}2^{-\frac d2 m}   2^{\frac d{2r} m-\frac dr j} 2^{\frac \alpha q m}  2^{-\frac \gamma q j} \le  C  \big\| U^\gamma_t P_j f\big\|_{L_t^q( E(2^{-\gamma j} \cap J); L_x^r)} .\]
Since $\|f\|_{H^s}\sim 2^{(s+\frac d2)j-\frac d4 m}$,  from  the estimate \eqref{f-stri} it follows that 
\[  2^{(d-\frac dr-\frac \gamma q) j} 2^{(\frac d{2r}+\frac\alpha q-\frac d2) m}   \le C  2^{(s+\frac d2)j-\frac d4 m}.   \] 
Using \eqref{con-reg}  with equality,  we obtain $2^{(\frac d{2r}+\frac\alpha q-\frac d4) m}   \le C$. Taking $m\to \infty$, we conclude that  \eqref{con-ad} holds.  
\end{proof}

\subsubsection{Necessity of \eqref{con-ad} and \eqref{con-reg2} for the estimate \eqref{f-stri10}}  
We now discuss sharpness of the estimate \eqref{f-stri10}.   Using the example in the proof of Proposition \ref{thm-nec}, it is easy to see that the estimate \eqref{f-stri1} holds true only if 
\eqref{con-reg2} is satisfied, provided that there is a point $t_\circ\in \supp \mu$ such that  
\begin{align}\label{alpha-dim-low}
	\mu( (t_\circ-\rho, t_\circ+ \rho)) \ge C\rho^\alpha
\end{align}
for $\rho\in (0, 1).$  Hence, we need only to focus on the necessity of \eqref{con-ad} for \eqref{f-stri10}.

\begin{prop}\label{prop:nec}  Let $\gamma\in (0, \infty) \setminus \{1\}$ and  \eqref{con-reg2} hold with equality.  Suppose that   the estimate  
\eqref{f-stri1}   holds  whenever  $\mu$ is a probability measure supported in $[0,1]$ satisfying \eqref{alpha-dim}.
Then, \eqref{con-ad} holds. 
 \end{prop}

\begin{proof}  We again use  Lemma \ref{ex}.   There are a set $E$ and  a measure $\mu$ with $\supp \mu=E$
satisfying \eqref{mess} and,  consequently,  \eqref{alpha-dim}.  As in the proof of Proposition \ref{nec-del},  taking $l=\lfloor \gamma j \rfloor$ and $k=m-L$ for a large positive integer $L$ and assuming $\gamma j\gg m$, we apply $\rm (A)$ in Lemma \ref{ex}.  Then, there is  an interval $J$  of length $ C_12^{m-\gamma j}$, which includes  $C_2 2^{\alpha m}$ elements of  a $C_32^{-\gamma j}$-separated subset $\tilde E$ of $E$. 

This means that $J$ contains disjoint intervals $J_1, \dots,  J_n$ of length $C_32^{-\gamma j-1}$  with $n\sim 2^{\alpha m}$, and  each of the intervals  $J_1, \dots,  J_n$ contains exactly one element of $\tilde E$.   Thus, by  \eqref{mess}, we have $\mu(J_k)\sim 2^{-\alpha\gamma  j} $ for $k=1, \dots, n$. Therefore, 
\[    \mu(J)\ge C 2^{\alpha m}  2^{-\alpha\gamma  j}.  \]
Combining this with \eqref{lower-}, we see that  the estimate \eqref{f-stri1} implies 
\[  2^{dj}2^{-\frac d2 m}   2^{\frac d{2r} m-\frac dr j} 2^{\frac\alpha  q m}  2^{-\frac{\alpha\gamma}q  j} \le C  2^{(s+\frac d2)j-\frac d4 m}. \] 
Consequently, using \eqref{con-reg2} with equality and letting $j\to \infty$ yield the condition \eqref{con-ad}. 
\end{proof}

\section{Inhomogeneous estimates}
\label{sec:inho}

In this section, we prove Theorem \ref{thm_inhom1}.   The proofs of the estimates  \eqref{f-inhom} and \eqref{inho-est} are almost identical, so we prove   
\eqref{inho-est} only.

For $j\in \mathbb N_0$,  consider an operator $\mathcal{U}_j$ given by
\[
	\mathcal{U}_jF(x,t) 
	= \int_{[0,t]} \mathrm{e}^{i(t - s)(-\Delta)^{\gamma/2}} P_j F(s)\,\mathrm{d}\mu (s).
\]
Since $2\le  \tilde r, r <\infty$ and $\tilde q, q\ge 2$, by the Littlewood--Paley inequality  
the  estimate \eqref{inho-est}  follows  if we show
\begin{align}\label{250529_0034}
	\big\|  \mathcal{U}_jF\big\|_{L_t^q(\mathrm{d}\mu; L_x^r)} \le C\| F\|_{L_s^{\tilde{q}'}(\mathrm{d}\mu ; L_x^{\tilde{r}'})}. 
\end{align}

\subsection*{Proof of the estimate  \eqref{250529_0034}}  
By duality, we have 
\[
	\| \mathcal{U}_j(F) \|_{L_t^q(\mathrm{d}\mu; L_x^r)} 
	= \sup_{G \in L_t^{q'}\!(\mathrm{d}\mu; L_x^{r'})} \iint_{\{ s :0<s\leq t\}} \langle U_s^\gamma  P_j F (s), U_t^\gamma \tilde P_j G(t)\rangle_x \,\mathrm{d}\mu (s) \mathrm{d}\mu(t).
\]
As in the endpoint case discussed before, we consider 
\[
	\mathfrak{T}(F,G) = \iint_{\{ s :0<s\leq t\}} \langle U_s^\gamma P_j F (s), U_t^\gamma \tilde P_j G(t)\rangle_x \,\mathrm{d}\mu (s) \mathrm{d}\mu(t).
\]
To obtain \eqref{250529_0034}, it suffices to obtain the equivalent  estimate 
\Be 
\label{bi}
|\mathfrak{T}(F,G)|  \le C 2^{j s}\| F\|_{L_s^{\tilde{q}'}(\mathrm{d}\mu ; L_x^{\tilde{r}'})}  \| G \|_{L_t^{q'}(\mathrm{d}\mu; L_x^{r'})}.
\Ee

Recalling \eqref{chi-k}, we  set
\begin{align*}
\mathfrak{T}_k(F,G)= \iint  \chi_k(t-s)\langle U_s^\gamma P_j F (s), U_t^\gamma \tilde P_j G(t)\rangle_x \,\mathrm{d}\mu (s) \mathrm{d}\mu(t).
\end{align*}
Thus, we have
\begin{align*}
	\mathfrak{T}(F,G) =\sum_{k\ge 0} \mathfrak{T}_k(F,G).
\end{align*}

Now, the proof of the estimate  \eqref{bi} basically reduces to  establishing  the following lemma.

\begin{lem}\label{lem_inhom1}  Let $\alpha \in (0,1]$ and $\sigma_{\alpha}^\gamma(\tilde{r}, r, \tilde{q}, q)$ be given by \eqref{con-inh-s}.   If $\alpha \neq d/2$, then  one has the estimate 
	\Be
	\label{Tkj}
		| \mathfrak{T}_k(F,G) | \leq C 2^{\sigma_{\alpha}^\gamma(\tilde{r}, r, \tilde{q}, q)j} 2^{- \lambda_{\alpha}(\tilde{r}, r, \tilde{q}, q)k} \|F\|_{L_s^{\tilde q'} (\mathrm{d}\mu; L_x^{\tilde r'})} \| G\|_{L_t^{q'}(\mathrm{d}\mu; L_x^{r'})}
	\Ee
	for  $(1/\tilde r, 1/r)\in \mathcal Q$ (see \eqref{aa}) and  $\tilde q, q$ satisfying  \eqref{ran_q_up} and \eqref{ran_q_low}, 
	where  
	\Be
\label{con-inh-lambda}  	 \lambda_{\alpha}(\tilde{r}, r, \tilde{q}, q) =\frac d2 \Big(1 - \frac1{\tilde{r}}-\frac1r \Big)  - \Big(\frac{\alpha}{\tilde{q}} + \frac\alpha q \Big). 
\Ee
	If $\alpha=d/2$, then  one has the estimate \eqref{Tkj}  for  $2\le \tilde r, r<\infty$ and  $\tilde q, q$ satisfying  \eqref{ran_q_up} and \eqref{ran_q_low}. 
\end{lem}

Since we are imposing the condition \ref{scaling} and $\gamma \ge 2$,  we have
\[ 
 \lambda_{\alpha}(\tilde{r}, r, \tilde{q}, q) =\Big(\frac \gamma 2-1\Big) \Big(\frac{\alpha}{\tilde{q}} + \frac\alpha q \Big)\ge 0.\]
Once  Lemma~\ref{lem_inhom1} is established,   the estimate \eqref{bi}  follows by direct summation when  $\lambda_{\alpha}(\tilde{r}, r, \tilde{q}, q)> 0$, 
and by the bilinear interpolation argument in Section~\ref{sec_endpoint} when  $\lambda_{\alpha}(\tilde{r}, r, \tilde{q}, q)=0$ except for the cases $(1/\tilde r, 1/r)=C,C'$ (see \eqref{cd}).

\begin{proof}  We consider the cases $\alpha< d/2$, $\alpha> d/2$, and  $\alpha=d/2$, separately.  We handle the case $\alpha< d/2$ first. 

We prove this case  by interpolating the estimate  \eqref{Tkj} for the cases $(1/\tilde{r}, 1/r)=(0,0),$ $(1/2, 1/2),$  $A=(1/2, ((d-2\alpha)/2d),$ $A'=((d-2\alpha)/2d,1/2)$. For the purpose, as in the previous section,  we may assume that $F(x, \cdot)$ and $G(x,\cdot)$ are supported in intervals of length $2^{k-\gamma j }$ since $t, s$ satisfy $|t-s|\sim 2^{k-\gamma j }$ and $\tilde q' \le q$.

First, let $r=\tilde{r} = \infty$.  
It is clear that 
$(U_t^\gamma \tilde P_j)^* U_s^\gamma P_j$ has the same bound as $T_{s,t}$ given by \eqref{utut}. 
Thus, by \eqref{gr} and \eqref{250309_1101} with $r=\infty$, one has
\begin{align*}
	|\mathfrak{T}_k(F,G)|
	& \le C  2^{dj} 2^{-\frac d2 k}  \iint_{|t-s|\le 2^{k-\gamma j }}\|F(s) \|_{L_x^1}  \| G(t)\|_{L_x^1}\,\mathrm{d}\mu(s) \mathrm{d}\mu(t) \\
	& \le C  2^{dj} 2^{-\frac d2 k} \|F\|_{L_s^1 (\mathrm{d}\mu; L_x^1)} \| G\|_{L_t^1(\mathrm{d}\mu; L_x^1)}.
\end{align*}
By  the $t$-support assumption on $F$ and $G$, H\"older's inequality, and the assumption \eqref{alpha-dim}, it follows that
\[
	|\mathfrak{T}_k(F,G)|
	\leq C 2^{dj} 2^{-\frac d2 k} \left(  2^{\frac{k - \gamma j }{\tilde{q}}{\alpha}} \| F \|_{L_s^{\tilde q'}(\mathrm{d}\mu ; L_x^1)}\right) \left(  2^{\frac{k - \gamma j }{q} \alpha}  \|G\|_{L_t^{q'}(\mathrm{d}\mu; L_x^1)}\right)
\]
for $q, \tilde{q}\geq 1$. Thus, we have
\Be
\label{0-0}
	|\mathfrak{T}_k(F,G)|
	\leq 
		C 2^{j(d-  \frac{\gamma{\alpha}}{\tilde{q}} - \frac{\gamma\alpha}{q}  )} 
		2^{-k(\frac d2 - \frac{\alpha}{\tilde{q}} - \frac{\alpha}{q}  )} 
		\| F \|_{L_s^{\tilde q'}(\mathrm{d}\mu ; L_x^1)}  \|G\|_{L_t^{q'}(\mathrm{d}\mu; L_x^1)}.
\Ee
Recalling  \eqref{con-inh-s} and \eqref{con-inh-lambda},  we see that the above inequality is just \eqref{Tkj} for $r=\tilde{r} = \infty$  and  $q, \tilde{q}\geq 1$. 

Secondly, let $r = \tilde{r} = 2$. By Plancherel's theorem,  one has
\begin{align*}
	|\mathfrak{T}_k(F,G)| 
	\leq &\| F \|_{L_s^1(\mathrm{d}\mu ; L_x^2)} \| G\|_{L_t^1(\mathrm{d}\mu; L_x^2)}.
\end{align*}
Thus, H\"older's inequality gives 
\Be 
\label{2-2}
	|\mathfrak{T}_k(F,G)|
	\leq  C  2^{-j(\frac{\gamma{\alpha}}{\tilde{q}} + \frac{\gamma\alpha}{q})} 2^{k( \frac{\alpha}{\tilde{q}}+\frac{\alpha}{q})}  \| F \|_{L_s^{\tilde{q}'}(\mathrm{d}\mu ; L_x^2)} \| G\|_{L_t^{q'}(\mathrm{d}\mu; L_x^2)}.
\Ee
 This shows \eqref{Tkj} for $r=\tilde{r} =2$  and  $q, \tilde{q}\geq 1$.  
 Indeed, from \eqref{con-inh-s} and \eqref{con-inh-lambda}, we  note that $
	\sigma_{\alpha}^\gamma(q, \tilde{q}, 2, 2) = -{\gamma{\alpha}}/{\tilde{q}} - {\gamma\alpha}/{q}$ and $\lambda_{\alpha}(q, \tilde{q}, 2, 2)=- {\alpha}/{\tilde{q}} -\alpha/{q}. $

Thirdly, let $\tilde{r}= {2d}/(d-2\alpha)$ and $r=2$. 
Note that \begin{align*} 	\mathfrak{T}_k(F,G) 	= \int \Big\langle  \int  U_s^\gamma P_j F_{k,t}(s)\,\mathrm{d}\mu (s), U_t^\gamma  \tilde P_j G(t) \Big\rangle\,\mathrm{d}\mu(t),  \end{align*}
where $F_{k, t} (s) =\chi_k(t-s) F(s)$. 
The Cauchy-Schwarz inequality in $x$-variable yields 
\begin{align*}
	|\mathfrak{T}_k(F,G)|  
	&\leq \sup_t \Big\| \int  U_s^\gamma P_j F_{k,t}(s) \,\mathrm{d}\mu (s)\Big\|_2   \int   \| G(t) \|_{L_x^2}\,\mathrm{d}\mu(t).
	\end{align*}
Using  the estimate \eqref{f-stri1}  from Theorem \ref{cor_main} with $\widetilde r = 2d/(d-2\alpha)$ and $q=2$ (actually, the dual form of the estimate), we have
	\begin{align*}
|\mathfrak{T}_k(F,G)|  
	&\leq  C\int 2^{(\alpha-\frac{\gamma\alpha}2)j} \| F \|_{L_s^2(\mathrm{d}\mu ; L_x^{\frac{2d}{d+2\alpha}})} \| G(t) \|_{L_x^2}\,\mathrm{d}\mu(t),
\end{align*}
where the exponent is determined by
\eqref{con-reg}.  Thus, 
\begin{align*}
	|\mathfrak{T}_k(F,G)| 
	\le C  &2^{(\alpha-\frac{\gamma\alpha}2)j} \| F \|_{L_s^2(\mathrm{d}\mu ; L_x^{\frac{2d}{d+2\alpha}})} \| G\|_{L_t^1(\mathrm{d}\mu; L_x^2)}.
	\end{align*}
		Since the $t$-supports of  $F$ and $G$ are contained in an interval $I$ of length $2^{k-\gamma j }$,  by  the assumption \eqref{alpha-dim}  and H\"older's inequality,   we obtain 
	\begin{align*} 
	|\mathfrak{T}_k(F,G)|  	\leq &2^{ (\alpha- \frac{\gamma{\alpha}} {\tilde q} - \frac {\gamma\alpha} q)j} 2^{-(\frac{\alpha}2 - \frac {\alpha} {\tilde q} -\frac \alpha q)k} \| F \|_{L_s^{\tilde q'}(\mathrm{d}\mu; L_x^{\frac{2d}{d+2\alpha}})} \| G\|_{L_t^{q'}(\mathrm{d}\mu; L_x^2)}
\end{align*}
 for $2 \le \tilde q\, ' \le q$.
Recalling   \eqref{con-inh-s} and \eqref{con-inh-lambda}, consequently, we see that this establishes the estimate \eqref{Tkj}  for $\tilde{r}= {2d}/(d-2\alpha)$ and $r=2$ for $2 \le \tilde q\, ' \le q$.

Now, we interpolate  the estimate \eqref{Tkj} for  $(1/\tilde r, 1/r)=(0,0), (1/2, 1/2)$, and $A'$, which are established above. Note that the condition  \eqref{ran_q_up} is satisfied 
for each case.  That is to say,   we have   \eqref{Tkj}  for $1\le \widetilde q\,'\le q\le\infty$ when $(1/\widetilde r, 1/r) = (0,0), (1/2, 1/2)$, and   for  $2\le \widetilde q \,' \le q\le\infty$  when $(1/\widetilde r, 1/r) = A'$.  Interpolation between those estimates yields  \eqref{Tkj}
 for  all  $(1/\tilde r, 1/r)\in \mathcal Q$ satisfying $1/r \ge 1/\tilde r$, provided that $\tilde q, q$ satisfy   \eqref{ran_q_up}.

\begin{comment}
Noting that \eqref{ran_q_up} is satisfied 
for each case,  via interpolation between those estimates,  one gets \eqref{Tkj}
 for  $(1/\tilde r, 1/r)\in \mathcal Q$ satisfying  $1/r \ge 1/\tilde r$ and 
$\tilde q, q$ satisfying   \eqref{ran_q_up}.     
\end{comment}

We now consider the case $\tilde{r} = 2$ and $r = \frac{2d}{d - 2\alpha}$. By symmetry, this case can be handled using the same argument as in the previous case by interchanging the roles of $F$ and $G$. That is, writing
\begin{align*}
	\mathfrak{T}_k(F,G) = \int \Big\langle U_s^\gamma P_j F(s), \int \chi_k(t - s) U_t^\gamma \tilde{P}_j G(t)\, \mathrm{d}\mu(t) \Big\rangle\, \mathrm{d}\mu(s),
\end{align*}
we obtain the estimate
\[
|\mathfrak{T}_k(F,G)| 
\le C\, 2^{(\alpha - \frac{\gamma\alpha}{2})j} \|F\|_{L_s^1(\mathrm{d}\mu; L_x^2)} \|G\|_{L_t^2(\mathrm{d}\mu; L_x^{\frac{2d}{d + 2\alpha}})}.
\]
Then, applying H\"older's inequality together with \eqref{alpha-dim}, we obtain the estimate \eqref{Tkj} for $\tilde{r} = 2$, $r = \frac{2d}{d - 2\alpha}$, and $1 \le \tilde q \,' \le q \le 2$. As before, by interpolating the estimates corresponding to the cases $(1/\tilde{r}, 1/r) = (0,0)$, $(1/2, 1/2)$, and $A$, we conclude that \eqref{Tkj} holds for all $(1/\tilde{r}, 1/r) \in \mathcal{Q}$ satisfying $1/r \le 1/\tilde{r}$, and for $\tilde{q}, q$ satisfying \eqref{ran_q_low}.  This completes the proof for the case $\alpha< d/2$.

Let us consider the case $\alpha> d/2$. To this end, we recall the  estimate \eqref{f-stri1p}  that holds  for $2\le q,r \leq \infty$ satisfying \eqref{con-ad} and  $(q,r)\neq (2,\infty)$.  Thanks to frequency localization,  \eqref{f-stri1p} remains valid for $(q_\circ,r)=(4\alpha/d, \infty)$.  Using this estimate similarly as above,  we have  
\[|\mathfrak{T}_k(F,G)| 
	\le C  2^{(\frac d2-\frac{\gamma\alpha}{q_\circ})j} \| F \|_{L_s^{q_\circ'}(\mathrm{d}\mu ; L_x^{1})} \| G\|_{L_t^1(\mathrm{d}\mu; L_x^2)}.\]
Thus, by H\"older's inequality and \eqref{alpha-dim}, we obtain \eqref{Tkj}  for $(1/\tilde r, 1/r)=(0, 1/2)$ and $0\le 1/q \le 1/\tilde q'\le 1-d/(4\alpha)$.  Note that 
 $\tilde q, q$ here satisfy   \eqref{ran_q_up} for $(1/\tilde r, 1/r)=(0, 1/2)$.  Thus, interpolation with the estimates \eqref{Tkj}  for $(1/\tilde r, 1/r)=(0,0)$ and $(1/2, 1/2)$ obtained above (i.e., the estimates \eqref{0-0} and \eqref{2-2})     gives the desired estimate for $(1/\tilde r, 1/r)\in \mathcal Q=[0,1/2]\times [0,1/2]$ satisfying $1/r \ge 1/\tilde r$, and  $\tilde q, q$ satisfying    \eqref{ran_q_up}. The case $(1/\tilde r, 1/r)\in \mathcal Q$ satisfying $1/r \le 1/\tilde r$ can be handled symmetrically as before. We omit the detail.

Finally, we consider the case $\alpha=d/2$. In this case, unfortunately,  the estimate \eqref{f-stri1p} is  not available for $(q, r)=(2,\infty)$.  
However, we can take $(q_\ast, r_\ast)$ such that $1/q_\ast +1/r_\ast=1/2$ and $2< q_\ast <r_\ast$, so \eqref{f-stri1p} holds for $(q, r)=(q_\ast, r_\ast)$. 
Using this estimate as above, one obtains  \eqref{Tkj} with $(1/\tilde r, 1/r)=(1/r_\ast, 1/2)$ for which $\tilde q, q$ satisfy  \eqref{ran_q_up}.  Interpolation with the estimates \eqref{Tkj}  for $(1/\tilde r, 1/r)=(0,0)$ and $(1/2, 1/2)$  (i.e., the estimates \eqref{0-0} and \eqref{2-2})  gives   \eqref{Tkj} for $(1/\tilde r, 1/r)$ contained in the triangle 
with vertices $(0,0), (1/2, 1/2),$ and $(1/{r_\ast}, 1/2)$,  and  $\tilde q, q$ satisfying    \eqref{ran_q_up}.  As before, a  symmetric argument gives \eqref{Tkj} for $(1/\tilde r, 1/r)$ contained in the triangle 
with vertices $(0,0), (1/2, 1/2),$ and $(1/2, 1/{r_\ast})$. Thus,  we get \eqref{Tkj} for $(1/\tilde r, 1/r)$ contained in the quadrangle with  vertices $(0,0), (1/2, 1/2),$ $(1/{r_\ast}, 1/2)$, and $(1/2, 1/{r_\ast})$. Now, taking $r_\ast\to \infty$, we obtain  get \eqref{Tkj} for $(1/\tilde r, 1/r)\in (0,1/2]\times (0, 1/2]$. This completes the proof. 
\end{proof}

Before closing this section, we make  remarks 
regarding the inhomogeneous estimates with  additional regularity, and similar estimates for wave equations.

\begin{rem}\label{rem:inhomo} 
If one allows additional regularity parameter in \eqref{f-inhom}, there are further available estimates. Indeed, consider 
\begin{align}\label{f-inhom-s}
	&\Big\| \int \mathrm{e}^{i(t - s)(-\Delta)^\gamma} F(s)\,\mathrm{d}\mu(s) \Big\|_{L_t^q(\mathrm{d}\mu; \dot L_x^{r,-\sigma})} \le C \| F\|_{L_s^{\tilde{q}'}(\mathrm{d}\mu ; \cdot \dot L_x^{\tilde{r}', \sigma })}
\end{align}
for some $\sigma $.  Here $\dot L^{p, s}_x$ denotes the homogeneous $L^p$ Sobolev space with regularity order $s$.  For the estimate \eqref{f-inhom-s} the scaling condition 
\[ \sigma_{\alpha}^\gamma(\tilde{r}, r, \tilde{q}, q) =2\sigma\]
is necessary. Moreover,   the inequality 
$\lambda_{\alpha}(\tilde{r}, r, \tilde{q}, q) \ge 0$ has to be satisfied,  instead of  simply imposing $\gamma\ge 2$. 
\end{rem}

\begin{rem}\label{inhomo-wave}   
One can obtain analogous results for the wave case, i.e., when $\gamma = 1$, since we have \eqref{dis-wave} and  \eqref{f-stri1} for $(q, r) \in [2, \infty] \times [2, \infty] \setminus \{(2, \infty)\}$ satisfying \eqref{con-reg} with $\gamma = 1$ and \eqref{con-ad-w}. In fact, by replacing $d$ with $d - 1$ in \eqref{aa}, \eqref{cd}, and \eqref{con-inh-lambda}, we define $
	\tilde{A} = (\frac{1}{2}, \frac{d - 1 - 2\alpha}{2(d - 1)} ),$  $\tilde{A}'$, as well as  $\tilde{C}, \tilde{C}'$, and  $\tilde{\lambda}_\alpha$, which replace $A, A', C, C'$, and $\lambda_\alpha$, respectively.  Then, a statement analogous to that of Theorem~\ref{thm_inhom1} holds. That is, under the same assumptions as in Theorem~\ref{thm_inhom1} with $d$ replaced by $d - 1$, the estimate \eqref{inho-est} with $\gamma = 1$ holds for all $\sigma  \ge \sigma_{\alpha}^1(\tilde{r}, r, \tilde{q}, q)$ and $\widetilde{\lambda}_\alpha(\widetilde r, r, \widetilde q, q) =  \frac{d-1}2 (1 - \frac1{\widetilde r} - \frac1r) - (\frac \alpha{\widetilde q} + \frac\alpha q)\ge0$.
\end{rem}

\section{$L^2$ local smoothing estimates}
\label{sec:l2-ls}

In this section, we provide the proofs of Theorems \ref{thm_LS} and \ref{thm_LS-homo}, which rely on a temporal localization lemma, Lemma \ref{lem_CLV} below.

\subsection{Proof of Theorem \ref{thm_LS}}  
Let $\omega\in \mathcal S(\mathbb R^d)$ such that $\supp \widehat \omega \subset  \mathbb B^d$  and   $\omega \ge 1$ on $\mathbb B^d$. 
Thus, the estimate  \eqref{ls-s} follows if we show 
\[\left\| \omega U_t^\gamma f \right\|_{L_{t}^2 (\rd\mu; L^2_x(\mathbb{B}^d))} \leq C \|f\|_{H^s}.\]
By the Littlewood--Paley decomposition and orthogonality between $ \omega U_t^\gamma P_jf$ as a function of the spatial variable,  it suffices to show 
$\left\| \omega U_t^\gamma  P_j f \right\|_{L_{t}^2 (\rd\mu; L^2_x(\mathbb{B}^d))} \leq C \|f\|_{H^s}$ 
for each $j$. Thanks to the rapid decay of $\omega$, this in turn follows from
\Be\label{lse_j} 	\| U_t^\gamma P_j f \|_{L_{t}^2 (\rd\mu; L^2_x(\mathbb{B}^d))}^2 \leq C2^{\alpha(1-\gamma)j} \|f\|_2^2.  \footnote{In fact, we need to apply this local estimate again using 
dyadic decomposition in the spatial variable and scaling. However, the desired estimate follows without difficulty by virtue of the rapid decay of $\omega$.}
\Ee

Let ${\mathcal J}$ be a collection  of disjoint intervals  $J\subset [0,1]$ of length $2^{(1-\gamma)j}$  which covers $E$, so 
\[ E\subset \bigcup_{J\in \mathcal J} J.\]
Thus, we have 
\Be 
\label{jjjj}
	\| U_t^\gamma P_j f \|_{L_{t}^2 (\rd\mu; L^2_x(\mathbb{B}^d))}^2 = \sum_{J \in \mathcal J}  \| U_t^\gamma P_j f\|_{{L_{t}^2 (\rd\mu, J; L^2_x(\mathbb{B}^d))}}^2.
\Ee

Now,  we recall a useful lemma from \cite{CLV}, which makes it possible to localize \eqref{lse_j} on time intervals of length $2^{(1-\gamma)j}$. 
Let $Q(\cdot, t)$ be a real valued smooth function satisfying 
\begin{align}
	\left| \nabla_\xi (Q(\xi, t) - Q(\xi, s)) \right| &\sim |t -s| |\xi|^{\gamma-1},\label{250629_1530}\\
	\left| \partial_\xi^\beta (Q(\xi, t) - Q(\xi, s)) \right| & \leq C |t-s| |\xi|^{\gamma - |\beta|} \label{250629_1531}
\end{align}
for $t,s \in [-2,2]$ and $|\xi| \geq1$. Let $T$ be defined  by
\[
	T f (x,t) = \int_{\mathbb{R}^d} \mathrm{e}^{2\pi i (x\cdot \xi + Q(\xi, t))} \beta( 2^{-j} |\xi|)  \widehat{f}(\xi)\,\mathrm{d} \xi. 
\]
Then, we have the following.

\begin{lem}[\cite{CLV}, Lemma 2.1]\label{lem_CLV} 
Let $q,r\geq2$, and  $\nu\in \mathbb{R}$. Let $\mu$ be a finite  measure supported on $[-2,2]$. 
For $j\ge 0$, let $\mathcal{J} = \{ J\}$ be a collection of disjoint intervals of length $2^{(1-\gamma)j}$  included in $[-2, 2]$. 
Suppose that \eqref{250629_1530} and \eqref{250629_1531} hold for $|\beta| \leq \max\{2, d-2\nu +3\}$.
Suppose also that
\begin{align*}
	\| T f \|_{L_t^q(\mathrm d\mu, J; L_x^r(\mathbb{B}^d))}  \leq C 2^{\nu j} \| f\|_2
\end{align*}
with a constant $C$ for all $J \in \mathcal{J}$.
Then, there exists a constant  $C$ such that
\Be
	\| T f \|_{ L_t^q(\mathrm d\mu, \cup_{J\in \mathcal{J}} J; L_x^r(\mathbb{B}^d))} \leq C  2^{\nu j}  \|f\|_2. \nonumber
\Ee
\end{lem}

The lemma was originally proved in \cite{CLV} with $\mathrm{d}t$ in place of $\mathrm{d}\mu$, and with the order of integration interchanged, using a standard $TT^*$ argument. Nevertheless, it is straightforward to verify that the same argument remains valid for any finite measure while the order of integration interchanged.

Here, we provide a brief sketch of the proof of Lemma \ref{lem_CLV}.  
For each $J\in \mathcal J$,  
set
\[F_J(x,t)=\chi_{\mathbb{B}^d\times J }(x,t)F(x,t).\] 
To prove Lemma \ref{lem_CLV}, by duality, it is enough to show
\[
\|\sum_{J \in \mathcal J}T^{*} F_J\|_{2}  \leq C 2^{\nu j} \|F\|_{L_t^{q'}(\mathrm d\mu; L_x^{r'})}, 
\]
provided that  
\[  \|T^{*} F_J\|_{2} \le   C 2^{\nu j} \|F\|_{L_t^{q'}(\mathrm d\mu; L_x^{r'})}, \quad \forall  J\in \mathcal J, \]
where $T^*$ denotes the adjoint operator of $T$. Note that $T^*$ is given explicitly by
\[
T^*F (x) = \iint_{\mathbb{R}^d} \mathrm{e}^{2\pi i (x\cdot \xi- Q(\xi, s))}  \psi( 2^{-j} |\xi|) \widehat{F(s)}(\xi)\,\mathrm{d} \xi \rd \mu(s).
\]

Since  the estimates are local in both time and space, the  key part of the argument  in \cite{CLV} is to 
establish the inequality
\Be
\label{key-est}
\|\chi_{J}TT^{*} F_{J'}\|_{L_t^{\infty}(\mathrm d\mu; L_x^{\infty}(\mathbb{B}^d))} \le C 
2^{dj}(1+2^{\gamma j}\dist(J, J'))^{-L}\|F_{J'}\|_{L_t^{1}(\mathrm d\mu; L_x^{1}(\mathbb{B}^d))}.
\Ee
for $\dist(J, J') \ge
C2^{(1-\gamma)j}$ with  $C$ large enough, where $L=2\vee (
d-2 \nu+2).$  Let 
\[  K(x,y,s,t) = \int
e^{i((x-y)\cdot \xi +
Q(\xi,t)-Q(\xi,s))}\psi^2( 2^{-j}{|\xi|})\, \rd\xi, \]
and  we note that 
$$TT^{*}F = \iint K(x,y,s,t)F(y,s)\,\rd y\rd \mu(s). $$ 
Thus, \eqref{key-est} follows from the estimate  $ |K(x,y,s,t)|\le
C2^{dj}(1+2^{\gamma j}|t-s|)^{-L}$  when $|t-s|\ge
C2^{(1-\gamma)j}$. This was already proved in \cite{CLV}, combining \eqref{250629_1530} and  \eqref{250629_1531},  by routine integration by parts ($L$
times).

\begin{proof}[Proof of \eqref{lse_j}]
Let $q=r=2$ and $Q(\xi, t) = t|\xi|^\gamma$. Note that  $Q(\xi, t) = t|\xi|^\gamma$ satisfies both \eqref{250629_1530} and \eqref{250629_1531}. 
By \eqref{jjjj} and Lemma~\ref{lem_CLV} with $\nu =\alpha(1-\gamma)$, \eqref{lse_j} follows if we show
\[
	\| U_t^\gamma P_j f\|_{{L_{t}^2 (\rd\mu, J; L^2_x(\mathbb{B}^d))}}^2 \leq C 2^{\alpha(1-\gamma)j} \langle \mu \rangle_\alpha \|f\|_2^2.
\]
By  Plancherel's theorem, it is clear that $\| U_t^\gamma P_j f\|_{L^2_x(\mathbb{B}^d))}\le \|f\|_2$. Then, by \eqref{alpha-dim} the desired inequality follows.
\end{proof}

\subsection{Proof of Theorem \ref{thm_LS-homo}} 
The proof of \eqref{ls-s-homo} is similar to that of \eqref{ls-s}. Indeed, it suffices to show 
\[   	\| U_t^\gamma P_j f \|_{L_{t}^2 (E(2^{-\gamma  j}); L^2_x(\mathbb{B}^d))}^2 \leq C2^{(\alpha-\gamma)j} \|f\|_2^2 .\]
Let $F\subset [-2,2]$ be a measurable set.   Considering  $\rd\mu= \chi_F \rd t$, we see that  Lemma \ref{lem_CLV}  also holds for the measure $\chi_F \rd t$. Thus, the desired inequality \eqref{ls-s-homo} follows if we show 
\Be 
\label{jj-jj} 
\| U_t^\gamma P_j f \|_{L_{t}^2 (E(2^{-\gamma  j})\cap J; L^2_x(\mathbb{B}^d))}^2 \leq C2^{(\alpha-\gamma)j} \|f\|_2^2 
\Ee
for any interval  $J$ of length $2^{(1-\gamma)j}$.  

Let $\mathcal{I}$ denote a collection of disjoint  intervals of length $2^{-\gamma j}$, which cover $E(2^{-\gamma  j})\cap J$. 
 It is clear that 
\[ \| U_t^\gamma P_j f\|_{L_{t}^2 (E(2^{-\gamma  j})\cap J; L^2_x(\mathbb{B}^d))} \le \sum_{I: E(2^{-\gamma  j})\cap J \cap I \not= \emptyset}  \| U_t^\gamma P_j f\|_{{L_{t}^2 ( I; L^2_x(\mathbb{B}^d))}}^2.\]
Since $[E]_{\alpha, \frac{\gamma-1}\gamma}<\infty,$  we have $\# \{I \in \mathcal I : E(2^{-\gamma  j})\cap J \cap I \not= \emptyset\}\lesssim  2^{\alpha j}$. By Plancherel's theorem, we also have $\| U_t^\gamma P_j f\|_{{L_{t}^2 ( I; L^2_x(\mathbb{B}^d))}}^2\le C2^{-\gamma j} \|f\|_2^2$.  
Combining those observations, we conclude \eqref{jj-jj}. 

\subsection{Sharpness of the regularity exponent}  Using the same argument  in the proof of Proposition \ref{prop:nec}, we can show sharpness of the smoothing orders in Theorems \ref{thm_LS} and \ref{thm_LS-homo}.

\begin{prop}\label{prop:nec2}  
Let $\gamma > 1$. Suppose that the estimate \eqref{ls-s} holds for every probability measure $\mu$ satisfying \eqref{alpha-dim}, then  $s \geq  
\alpha(1-\gamma)/2$. 
Similarly, if the estimate \eqref{ls-s-homo} holds for every set $E \subset [0,1]$ with bounded $\alpha$-Assouad characteristic, then  $s \geq  
(\alpha-\gamma)/2$. 
\end{prop}

\begin{proof} 
To prove the first statement, by Lemma \ref{lem_CLV}, it suffices to consider  the estimate  over intervals $J$ of length $c2^{(1-\gamma)j}$ with $c$ determined later, while the Fourier support of the input function $f$ is contained in a ball of radius $2^{j/2}$ lying in the annulus $\{ \xi: 2^{j-1}\le |\xi|\le 2^{j+1}\}$.  We consider 
\[ \widehat{f}(\xi) = \psi( 2^{\frac j2}| 2^{-j}\xi-\xi_0|) e^{-it_\circ|\xi|^\gamma}.\]
Then, by the argument in the proof of Proposition \ref{nec-del} (taking $m=j$), we have 
\Be
\label{lower-2} |U^\gamma_t f(x)|\gtrsim 2^{\frac d2j}
\Ee
if 
$ |t-t_\circ|\le c2^{(1-\gamma) j}$ and $|x- 2^{(\gamma-1) j}(t-t_\circ) \nabla\phi(\xi_0)\big)| \le c2^{-j/2}$  for a sufficiently small $c>0$.

We make use of  Lemma \ref{ex}, so there are a set $E$ and  a measure $\mu$ with $\supp \mu=E$
satisfying \eqref{mess}. Clearly, $E$ has bounded $\alpha$-Assouad characteristic.  Taking $l=\lfloor \gamma j \rfloor$ and $k=j$, we apply $\rm (A)$ in Lemma \ref{ex}.  Then, similarly as in the proof of Proposition \ref{prop:nec}, there are  disjoint intervals $J_1, \dots,  J_n\subset J$ of length $C2^{-\gamma j}$  with $n\sim 2^{\alpha j}$, and  each of the intervals  $J_1, \dots,  J_n$ contains at least one element of $E$.   Thus, using \eqref{mess},    we obtain 
\[    \mu(J)\ge 2^{\alpha j}  2^{-\alpha\gamma  j}.  \]
Since $\|f\|_{H^s}\le C 2^{(s + \frac d4)j} $, by combining this and \eqref{lower-2} with  $t_\circ\in J$, the estimate \eqref{ls-s} implies 
\[  2^{\frac d2 j}  2^{-\frac d4 j} 2^{\frac\alpha  2j}  2^{-\frac{\alpha\gamma}2  j} \le C 2^{(s + \frac d4)j} . \] 
Taking $j\to \infty,$ we conclude that $s\ge \alpha(1-\gamma)/2 $.

We now show the second statement.  Since $\supp f\subset \{ \xi: 2^{j-1}\le |\xi|\le 2^{j+1}\}$, with suitable choices of $\psi$, one can easily see that 
\[  |U^\gamma_t P_j f(x)|\gtrsim 2^{\frac d2j}\]
if 
$ |t-t_\circ|\le c2^{(1-\gamma) j}$ and $|x- 2^{(\gamma-1) j}(t-t_\circ) \nabla\phi(\xi_0)\big)| \le c2^{-j/2}$  for a sufficiently small $c>0$. 
We use the intervals  $J_1, \dots,  J_n\subset J$ from the above, which contain at least one element of $E$. Thus, $|E(2^{-\gamma j})\cap J|\gtrsim 2^{(\alpha-\gamma) j} $.  Consequently, the estimate  \eqref{ls-s-homo} implies 
\[  2^{\frac d2 j}  2^{-\frac d4 j} 2^{(\frac\alpha  2-\frac \gamma 2)j} \le C 2^{(s + \frac d4)j} . \] 
Taking $j\to \infty$ yields  $s\ge (\alpha-\gamma)/2 $. 
\end{proof}

\section*{Acknowledgement} 
This work was supported by the National Research Foundation of Korea (RS-2024-00342160; S. Lee), (RS-2021-NR061906, RS-2024-00461749; J. B. Lee);  
the Basque Government  and Spanish MICIN and MICIU (BERC programme and Ikerbasque,  CEX2021-001142-S, PID2023-146646NB-I00, CNS2023-143893; Roncal).
The first and third authors would like to thank  F. Zhang and S. Zhao  for related discussions in the course of  the earlier work \cite{LRZZ}, which have benefited the current project.

\bibliographystyle{plain}

\end{document}